\documentclass{amsart}

\usepackage{amssymb,amsfonts}
\usepackage[all,arc]{xy}
\usepackage{hyperref}
\usepackage{mathrsfs}
\usepackage{physics}
\usepackage{tikz}
\usepackage{comment}
\usepackage{caption}
\usepackage{graphicx}
\usepackage{float}
\usepackage{bbm}
\usepackage{mathtools}

\usepackage[shortlabels]{enumitem}

\def\*#1{\mathbf{#1}}

\newcommand\eps{\varepsilon}

\DeclareMathOperator{\id}{id}

\DeclareMathOperator{\Vol}{Vol}

\DeclareMathOperator{\Cay}{Cay}

\newcommand\RR{\mathbb{R}}
\newcommand\ZZ{\mathbb{Z}}

\newtheorem{thm}[equation]{Theorem}
\newtheorem{cor}[equation]{Corollary}
\newtheorem{prop}[equation]{Proposition}
\newtheorem{lem}[equation]{Lemma}

\newtheorem{quest}[equation]{Question}
\newtheorem{mainthm}{Theorem}

\theoremstyle{definition}
\newtheorem{defn}[equation]{Definition}

\newtheorem{exmp}[equation]{Example}

\theoremstyle{remark}
\newtheorem{rmk}[equation]{Remark}

\numberwithin{equation}{section}

\title{Volume Entropy Rigidity for Random Groups at Low Densities
}

\author[D. (M.) Hua]{Dongming (Merrick) Hua}
\email{dongming@ucsb.edu}
\begin{document}
\begin{abstract}
    We study the rigidity of the volume entropy for weighted word metrics on hyperbolic groups, building on a recent convexity result due to Cantrell-Tanaka. Using ideas from small cancellation theory, we give conditions under which a hyperbolic group admits a unique normalized weight minimizing the entropy. Moreover, we show that these conditions are generic for random groups at small densities, and that the unique minimizer of such a generic group is arbitrarily close to the uniform weight.
\end{abstract}
\maketitle

\section{Introduction}
Given an inequality bounding some geometric quantity, we are often interested in the \textit{rigidity} of the inequality---whether there exists a \textit{unique} maximizing/minimizing object achieving equality. Perhaps the most famous example is the classical isoperimetric inequality: the disc $D^{2}$ minimizes the perimeter among all domains in the plane with the same area. A more recent example is the celebrated volume entropy inequality of G. Besson, G. Courtois, and S. Gallot \cite{BessonCourtoisGallot}: if $(M^{n}, g_{0})$, $n \geq 3$, is a closed hyperbolic manifold, then for any Riemannian metric $g$ on $M$ with $\Vol(M, g) = \Vol(M, g_{0})$ we have the inequality 
$$h(g) \geq h(g_{0}),$$ with equality if and only if $g$ is isometric to $g_{0}$. Here $h(g)$ is the \textit{volume entropy} of $(M, g)$, defined by 
$$h(g) \coloneqq \lim_{R \to \infty} \frac{\ln \Vol_{\widetilde{g}}(\widetilde{B}_{\widetilde{g}}(\widetilde{x}, R))}{R},$$ where the balls are taken in the universal cover $(\widetilde{M}, \widetilde{g})$. This limit is independent of the basepoint $\widetilde{x} \in \widetilde{M}$, as shown by A. Manning in \cite{manningbasepoint}.

The volume entropy can be defined for many classes of spaces other than Riemannian manifolds, including buildings \cite{ledrappier}, simplicial complexes \cite{babenkosabourau}, metric measure spaces \cite{rcdmaxentropy, rcdminentropy}, and finite weighted graphs \cite{lim}. Moreover, rigidity results have also been obtained in some of these settings, such as in \cite{rcdminentropy} or \cite{lim}.

In addition to the above examples, the volume entropy has also been studied in the context of groups. In the most basic case, given a group $G$ and a finite generating set $S$, one can consider the growth rate of the size of balls in the word metric with respect to $S$. Such a quantity has its roots in the work of A. Schwarz (also known as \v{S}varc) \cite{schwarz} and J. Milnor \cite{milnor} on the fundamental group of manifolds, and has since developed into an independent quantity of interest in geometric group theory; we refer the reader to \cite{delaharpe, grigorchukdelaharpe} for related problems and examples. 

More recently, there has been increased interest in groups equipped with more general metrics than just word metrics. Specifically, given a non-elementary hyperbolic group $G$, one can consider the set $\mathcal{D}(G)$ of all hyperbolic metrics which are left-invariant and quasi-isometric to some word metric; this perspective was first adopted by A. Furman in \cite{furman}, and has since been expanded on in \cite{cantrellreyes, cantrelltanaka, reyes}, just to list a few examples. The volume entropy of such metrics has been an important quantity in these papers, being related to Patterson-Sullivan measures for these metrics, serving as a normalizable quantity for comparing these metrics, and more.

This brings us to a natural question: \textit{are there volume entropy rigidity results for metrics in} $\mathcal{D}(G)$?

In our paper we consider a special case of the above question, studying the rigidity of the volume entropy of \textit{weighted marked groups} $(G, S, w)$, where $S$ is a finite, symmetric generating set for $G$ and $w \colon S \to (0, \infty)$ is a symmetric weight, satisfying $w(s) = w(s^{-1})$ for all $s \in S$. This weight $w$ induces a \textit{weighted word metric} on $G$, defined by
$$d_{w}(x, y) \coloneqq \min\left\{\sum_{i=1}^{n} w(s_{i}) \colon s_{1}, \ldots, s_{n} \in S, s_{1}\cdots s_{n} = y^{-1}x\right\}.$$ When $G$ is a non-elementary hyperbolic group, these metrics are all in $\mathcal{D}(G)$. We define the volume entropy of $w$ by
$$h(w) \coloneqq \lim_{R \to \infty} \frac{\ln |B_{w}(R)|}{R},$$ where $B_{w}(R)$ is the closed ball of radius $R$ in the metric $d_{w}$ centered at the identity. 

Our choice to restrict to these specific metrics is primarily motivated by the ability to normalize them by the sum of the weights of generators; since scaling metrics can make the entropy arbitrarily small, any study of minimizing metrics must require normalization, yet in general it is not clear which other quantities of metrics in $\mathcal{D}(G)$ we should normalize. Moreover, one can view this as a generalization of the work of S. Lim in \cite{lim}, which establishes the uniqueness of a minimizing normalized weight for finite, weighted graphs where each vertex has degree $\geq 3$.

Let us remark that A. Manning has also studied weighted word metrics, showing that the volume entropy of a manifold can be approximated by the volume entropies of weighted word metrics on the fundamental group \cite{manningsup}. However, A. Manning allows the generating set $S$ to vary whereas we do not.

To see that the entropy is well-defined, let $w_{\max} = \max_{s \in S}w(s)$ and $R, R' \geq 0$ be arbitrary. We see that any word with total weight bounded above by $R+R' + w_{\max}$ can be written as a concatenation of words with weights bounded above by $R + w_{\max}$ and $R' + w_{\max}$, respectively. Therefore, we have 
$|B_{w}(R+R'+w_{\max})| \leq |B_{w}(R+w_{\max})||B_{w}(R'+w_{\max})|$, and by the continuous version of Fekete's lemma \cite[Theorem 7.6.1]{hille} the limit 
$$\lim_{R \to \infty} \frac{\ln |B_{w}(R+w_{\max})|}{R} = \lim_{R \to \infty} \frac{\ln |B_{w}(R)|}{R}$$
exists and is finite.

We restrict $G$ to be a torsion-free, non-elementary hyperbolic group, and we assume the weights to be \textit{normalized}, meaning $\sum_{s \in S}w(s) = 1$. Both these conditions together guarantee the existence of a minimizing normalized weight. As with the previously studied settings, we are interested in conditions on $(G, S)$ which ensure the uniqueness of such a minimizing weight.  

The basis for our work is a recent convexity result for the volume entropy of more general metrics on hyperbolic groups, shown by S. Cantrell and R. Tanaka in \cite{cantrelltanaka}. To elaborate, they prove that the volume entropy is strictly convex up to \textit{rough isometry}, a natural notion of equivalence for metrics on groups. Moreover, they show a length spectrum rigidity result: two metrics are roughly isometric iff the translation lengths of all elements are equal. We refer the reader to Section \ref{sec: preliminaries} and Section \ref{sec: existenceanduniqueness} for further details and references. 

Our main contribution is the notion of a $\lambda$-\textit{translation-apparent presentation} for a group $G$. Such a presentation satisfies the classical $C'(\lambda)$ small cancellation condition, along with added conditions which guarantee that the generators are``evenly distributed" within each relator of the presentation. 

The key property of such presentations, as suggested by the name, is that the translation length of any generator $a$ with respect to any weight $w$ is equal to $w(a)$. So, we see that distinct weights give rise to distinct translation lengths, which along with the results of S. Cantrell and R. Tanaka implies the following theorem. 

\begin{mainthm}[Rigidity]\label{thm: mainthm1}
    Let $G$ be a torsion-free hyperbolic group admitting a $\lambda$-translation-apparent presentation $\langle A | \mathcal{R}\rangle$. Then $(G, S)$ admits a unique normalized weight minimizing the volume entropy, where $S = A \cup A^{-1}$.
\end{mainthm}

The proof of this theorem is inspired by a paper of A. Shukhov, which establishes bounds on the (unweighted) growth of certain small cancellation groups \cite{shukhov}. The small cancellation condition is used to show certain words are (unweighted) geodesics, as any other representative for the same group element would contain a large portion of a relator and thus fail to be geodesic. This, however, fails in the weighted case, since a portion of a relator which is large in the unweighted sense can still have extremely small weight. The even distribution condition in the definition of a translation-apparent presentation eliminates this possibility.

One potential concern is that the existence of a translation-apparent presentation may be a hard condition to satisfy. For example, the standard presentations for common hyperbolic groups such as surface groups are not translation-apparent for any $\lambda$ (although we do prove surface groups admit unique minimizers via other methods in Section \ref{sec: existenceanduniqueness}). 

To address this, we turn to M. Gromov's density model of random groups: given a length parameter $\ell$ and a density $0 \leq d < 1$, we randomly choose $(2m-1)^{d\ell}$ cyclically reduced words of length $\ell$ on a given alphabet of size $m$, and consider the random group obtained by taking these words as relators. The \textit{genericity} of a property $P$ means that the probability a random group in this model has property $P$ approaches $1$ as $\ell \to \infty$. 

Many important group properties are known to be generic at various densities, such as hyperbolicity for $d < \frac{1}{2}$; one can view this as a justification of the well-known heuristic that ``most groups are hyperbolic." Along these lines, our second main result states that most groups at low densities admit translation-apparent presentations.

\begin{mainthm}[Genericity]\label{thm: mainthm2}
   For each $m \geq 2$ there exists $0< d_{m} < 1$ depending on $m$, such that a generic random group on $m$ letters at density $d< d_{m}$ admits a translation-apparent presentation. In fact, this presentation can be obtained by symmetrizing the natural presentation given by the random relators.
\end{mainthm}
The essential tool in proving this result is a Chernoff bound for Markov chains, which we use to establish the even distribution condition for randomly chosen words. An explicit value of $d_{m}$ will be calculated in Section \ref{sec: randomgroups}. The reason we refer to such a setting as ``low" density is due to the fact that each $d_{m}$ is extremely small, and $d_{m} \to 0$ as $m \to \infty$.  We remark that regardless of $m$, this situation always includes the \textit{few-relator} model of random groups, where $d = 0$. Typically this consists of picking a constant number of random relators, although our results are still true if we interpret $d = 0$ to mean that we select a subexponential number (in terms of $\ell$) of random relators at each step $\ell$.

As a consequence of the previous two theorems, we conclude that a generic random group on $m$ letters at density $d < d_{m}$ has a unique normalized minimizer (with respect to the natural generating set). The natural next question is: what can we say about this unique minimizer, or the minimum value of the entropy? 

This brings us to our third result: generically, the entropy of any fixed weight is arbitrarily close to the entropy of the corresponding free group with the same weight. Moreover, the unique minimizer is arbitrarily close to being uniform (assigning an equal weight to each generator), and the minimum entropy is arbitrarily close to that of the corresponding free group.
\begin{mainthm}[Almost-Uniformity]\label{thm: mainthm3}
    Fix $\eps > 0$. A generic random group $G$ on $m$ letters at density $d < d_{m}$ satisfies $h(F_{m}, w) - \eps \leq h(G, w) \leq h(F_{m}, w)$ for \textit{any} normalized weight $w$. Moreover, $G$ admits a unique normalized minimizing weight $w_{0}$ satisfying: 
    \begin{enumerate}
        \item $\norm{w_{0}- w_{*}} < \eps$, where $w_{*}$ is the uniform normalized weight on the generators.
        \item The minimum entropy of $G$ is within $\eps$ of the minimum entropy of $F_{m}$.
    \end{enumerate}
\end{mainthm}
As above, we take the natural generating set for these random groups. We view each weight as an element of $\mathbb{R}^{|S|}$, hence the notation $\norm{w - w_{*}}$. The first part of this theorem can be viewed as a weighted generalization of a result of Y. Ollivier, which states that a generic random group has (unweighted) volume entropy arbitrarily close to that of the free group \cite{olliviergrowth}. 

The proof of our theorem relies on ideas from the aforementioned paper of A. Shukhov \cite{shukhov}, as well as a paper by A. Myers \cite{myers}. The former motivates us to bound the entropy by calculating the growth rate of a class of words which do not contain a large portion of a generator, and the latter gives us generating functions for weighted subword avoidance which allow us to carry out the necessary calculations.

\subsection*{Outline} Let us briefly discuss the structure of the paper. In Section \ref{sec: preliminaries} we collect some of the preliminaries needed for the later results. 

In Section \ref{sec: existenceanduniqueness} we show existence of a minimizing weight for any torsion-free, non-elementary hyperbolic group $(G, S)$. Then, we use the work of S. Cantrell and R. Tanaka \cite{cantrelltanaka} to show strict convexity of the volume entropy up to rough isometry, which implies the uniqueness of the minimizer up to rough isometry as well. As an application, we show that the uniform weight is the unique minimizer for surfaces groups with their usual generating set.

In Section \ref{sec: translationapparent} we define translation-apparent presentations, and show Theorem \ref{thm: mainthm1} using techniques from small cancellation theory. Moreover, we use the counting framework developed by A. Myers in \cite{myers} to establish inequalities for the volume entropy, which lays the foundation for Theorem \ref{thm: mainthm3}.

Moving on to random groups, in Section \ref{sec: randomgroups} we apply a Chernoff bound for Markov chains to show that the conditions in Section \ref{sec: translationapparent} are generic at low densities, giving us Theorem \ref{thm: mainthm2}. Furthermore, we apply the inequalities from Theorem \ref{thm: mainthm1} to prove Theorem \ref{thm: mainthm3}.

In Section \ref{sec: stability} we discuss the notion of stability for the volume entropy of weighted word metrics---we want to know if closeness to the minimum entropy implies closeness to the unique minimizer, in the case that we have rigidity. We prove that this is indeed the case.

In Section \ref{sec: futurework} we comment on some potential directions for future work. Finally, in Appendix\ref{sec: appendix} we prove some consequences of A. Myers' work which we previously omitted for reasons of readability.

\subsection*{Acknowledgements} Portions of this work were completed while the author was at the 2024 SURF program at Caltech. The author would like to thank Antoine Song for his invaluable mentorship and advice, as well as Caltech for supporting this work during the summer. The author would also like to thank Stephen Cantrell and Tom Hutchcroft for helpful discussions, as well as Fedya Manin and Alex Wright for comments and questions which improved the content and presentation of the paper.

\section{Preliminaries}\label{sec: preliminaries}
\subsection{Notation}
Given a group $G$, a group presentation for $G$ is a set of generators $A = \{a_{1}, \ldots, a_{m}\}$ along with a set of relators $\mathcal{R} \subseteq F(A)$, such that 
$$G \cong F(A)/\langle\langle \mathcal{R} \rangle \rangle.$$ Here $F(A) \cong F_{m}$ is the free group on the alphabet $A$, and $\langle \langle \mathcal{R} \rangle \rangle$ is the normal closure of $\mathcal{R}$. We will simply write $G = \langle A | \mathcal{R} \rangle$ to denote a group presentation. In this paper we will only deal with finite presentations, where both $A$ and $\mathcal{R}$ are finite.

Given a presentation $G = \langle A | \mathcal{R} \rangle$, we will always take the symmetric generating set $S = A \cup A^{-1} = \{a_{1}, \ldots, a_{m}, a_{1}^{-1}, \ldots, a_{m}^{-1}\}$ when studying the volume entropy. When the context is clear, we will often drop the generating set $S$ or even the group $G$ from the notation $h(G, S, w)$ for the volume entropy.

Words in $A$ (or in $S$) are simply strings $s_{1} \cdots s_{n}$ of elements $s \in S$. A word is \textit{reduced} if there do not exist $s_{i}, s_{i+1}$ which are mutual inverses. Reduced words are precisely the elements of $F(A)$. We further say that a word is \textit{cyclically reduced} if it is reduced and moreover $s_{1}, s_{n}$ are not mutual inverses (so a cyclically reduced word is a word whose cyclic permutations are all reduced as well).

There is a natural evaluation map from the set of words in $A$ to elements of $G$, given by first reducing the words to obtain words in $F(A)$, and then evaluating under the quotient $F(A) \to F(A)/\langle \langle \mathcal{R} \rangle \rangle \cong G$. We denote this evaluation map by $u \mapsto \bar{u}$. We say a word $u$ represents $x \in G$ if $\overline{u} = x$. Given a weight $w$, we say that a word $u$ is a $w$-geodesic if $u$ has shortest total weight (with respect to $w$) among all words representing $x = \overline{u}$. Observe that any $w$-geodesic must be reduced. Given a weight $w$, we denote by $|u|_{w}$ the total weight of $u$ with respect to $w$, and $|u|$ the unweighted word length of $u$. Observe that $d_{w}(e, x) = |u|_{w}$, where $u$ is a $w$-geodesic representative for $x$.

\subsection{Hyperbolic Metrics and Groups}\label{sec: hyperbolicgroups}
Hyperbolic groups and hyperbolic metrics on such groups are the main setting in which our results exist. We refer the reader to \cite{drutukapovich} as a standard reference for the material in this section. 

Given a metric space $(X, d)$, we say the metric $d$ is \textit{hyperbolic}  if it satisfies the following four point condition for some $\delta > 0$: for any $x, y, z, p \in X$ we require 
$$(x|y)_{p} \geq \min\left((x|z)_{p}, (y|z)_{p}\right) - \delta,$$ where the \textit{Gromov product} $(x|y)_{p}$ is defined by $$(x|y)_{p} = \frac{1}{2}(d(x, p) + d(y, p) - d(x, y)).$$ A finitely-generated group $G$ is a \textit{hyperbolic group} if for some (equivalently any) finite, symmetric generating set $S$, the word metric with respect to $S$ is a hyperbolic metric. 

We say a hyperbolic group $G$ is \textit{non-elementary} if it is infinite and does not contain $\ZZ$ as a finite-index subgroup. Given a non-elementary hyperbolic group $G$, we denote by $\mathcal{D}(G)$ the space of all metrics on $G$ which are left-invariant, hyperbolic, and quasi-isometric to some word metric on $G$. Recall that two metric spaces $(X, d), (Y, d')$ are \textit{quasi-isometric} if there exist constants $C > 1, K > 0$, and a function $f \colon X \to Y$ (called a quasi-isometry)  such that 
$$\frac{1}{C}d(x, x') - K \leq d'(f(x), f(x')) \leq Cd(x, x')+K,$$ and for any $y \in Y$ there exists $x \in X$ such that $d'(f(x), y) \leq K$. The space $\mathcal{D}(G)$ is closed under multiplicative scaling, as well as addition (\cite[Lemma 4.1]{reyes}), and in particular is closed under taking convex combinations. Moreover, this space contains all weighted word metrics:
\begin{prop}\label{prop: weightedishyperbolic}
If $G$ is a non-elementary hyperbolic group with a finite, symmetric generating set $S$, then for any symmetric weight $w \colon S \to (0, \infty)$ the weighted word metric $d_{w}$ is in $\mathcal{D}(G)$. 
\end{prop}
\begin{proof}
    The left-invariance of $d_{w}$ is clear from the definition. Consider the Cayley graph $\Cay(G, S)$, which has vertices corresponding to the elements of $G$, and an edge between $x, y \in G$ if and only if $x^{-1}y \in S$. We obtain two geodesic metrics $d_{1}, d_{2}$ on $\Cay(G, S)$, the first being the standard graph metric, and the second being the metric obtained by viewing the edge between $x$ and $y$ as an isometric copy of an interval of length $w(x^{-1}y)$. 
    
    Let $d_{S}$ be the standard word metric on $G$ with respect to $S$. Observe that $(G, d_{S})$ embeds isometrically into $(\Cay(G, S), d_{1})$ as the set of vertices, as does $(G, d_{w})$ in $(\Cay(G, S), d_{2})$. 
    
    Let $w_{\min} = \min_{s \in S} w(s)$, $w_{\max} = \max_{s\in S}w(s)$. We see that for any $x, y \in \Cay(G, S)$, we have 
    $$w_{\min}d_{1}(x, y) \leq d_{2}(x, y) \leq w_{\max}d_{1}(x, y),$$ and so the identity $\id \colon (\Cay(G, S), d_{1}) \to (\Cay(G, S), d_{2})$ is a quasi-isometry. Here, $(\Cay(G, S), d_{1})$ is hyperbolic because $G$ is a hyperbolic group.  Since quasi-isometries preserve hyperbolicity for geodesic metric spaces, we conclude that $(\Cay(G, S), d_{2})$ is hyperbolic. Moreover, $(G, d_{w})$ embeds isometrically in $(\Cay(G, S), d_{2})$, so the four point condition is satisfied and thus $d_{w}$ is a hyperbolic metric. The same considerations also show that $\id \colon (G, d_{S}) \to (G, d_{w})$ is a quasi-isometry, so $d_{w} \in \mathcal{D}(G)$ as claimed.
\end{proof}

\subsection{Volume Growth and Manhattan Curves.}
Let $G$ be a non-elementary hyperbolic group. Following Cantrell-Tanaka \cite{cantrelltanaka}, we define the \textit{volume growth rate} of $d \in \mathcal{D}(G)$ by
$$h(d) \coloneqq \limsup_{R \to \infty} \frac{\ln|B_{d}(R)|}{R}.$$ Observe that $h(d_{w}) = h(G, w)$ for weighted word metrics. The main tool used in \cite{cantrelltanaka} to compare the volume growth rate of two metrics $d, d' \in \mathcal{D}(G)$ is the \textit{Manhattan curve} $\mathcal{C}(d, d')$, defined by
$$\mathcal{C}(d, d') \coloneqq \partial \left\{(a, b) \in \RR \colon \sum_{x \in G}e^{-ad'(e, x) - bd(e, x)} < \infty \right\}.$$ Note that their initial definition of the Manhattan curve is different, but is equivalent to the one stated here by \cite[Proposition 3.1]{cantrelltanaka}.

Define $\theta_{d, d'}$ to be the function such that $(a, \theta_{d, d'}(a)) \in \mathcal{C}(d, d')$ for each $a$; in other words, such that $\mathcal{C}(d, d')$ is the graph of $\theta_{d, d'}$.  Then $\theta_{d, d'}(0) = h(d),$ $\theta_{d, d'}(h(d')) = 0$, and $\theta_{d, d'}$ is convex (see the paragraph following \cite[Theorem 1.1]{cantrelltanaka} and \cite[Corollary 2.10]{cantrelltanaka}). Moreover, we have:
\begin{thm}[{\cite[Theorem 1.1]{cantrelltanaka}}]\label{thm: manhattanline}
    The Manhattan curve $\mathcal{C}(d, d')$ is a straight line if and only if $d$ and $d'$ are roughly similar.
\end{thm}
Here, we say that $d, d'$ are \textit{roughly similar} if there exists $\lambda, C > 0$ so that $|d(x, y) - \lambda d'(x, y)| < C$ for all $x, y \in G$. When $\lambda = 1$, we say $d, d'$ are \textit{roughly isometric}. One characterization of roughly isometric metrics is via the \textit{(stable) translation length}, which is defined by
$$\ell_{d}(x) \coloneqq \lim_{n \to \infty}\frac{d(e, x^{n})}{n}$$ for $d \in \mathcal{D}(G)$ and $x \in G$.
\begin{thm}\label{thm: translationisometric}
    Let $G$ be a non-elementary hyperbolic group. Two metrics $d, d' \in \mathcal{D}(G)$ are roughly isometric if and only if for any $x \in G$, 
    $$\ell_{d}(x) =\ell_{d'}(x).$$
\end{thm}
This result follows from combining Theorem \ref{thm: manhattanline} with \cite[Lemma 3.4]{cantrelltanaka}.

\begin{rmk}
    Cantrell and Tanaka use $v(d)$ to denote the volume growth; we choose to use $h$ instead to fit with our notation for volume entropy. Notation aside, let us also note that in \cite[Proposition 2.3]{cantrellreyes}, Cantrell and Reyes prove a stronger result: the Manhattan curve is strictly convex unless the two metrics are roughly similar. However, we will not need this stronger statement. 
\end{rmk}
\subsection{Small Cancellation Theory}
Small cancellation theory deals with situations where the overlaps between relators in a group presentation are relatively small compared to their lengths. Techniques from this area will be our main tool for establishing uniqueness of minimizers. We refer the reader to \cite{lyndonschupp} as a standard reference.

Let $G = \langle A | \mathcal{R} \rangle$ be a finite presentation for $G$. We will assume that $\mathcal{R}$ consists of \textit{cyclically reduced} words, meaning that any cyclic permutation of $r \in \mathcal{R}$ is still a reduced word. Moreover, we will often pass to the \textit{symmetrization} $\mathcal{R}^{*}$ of $\mathcal{R}$, the closure of $\mathcal{R}$ under taking cyclic permutations and inverses. Note that cyclically reducing a word or adding a cyclic permutation or inverse of a word does not change the normal closure, so $\langle A | \mathcal{R}^{*} \rangle$ remains a presentation for $G$. We call such a presentation \textit{symmetrized}.

\begin{defn}\label{def: piece}
    A \textit{piece} with respect to a group presentation $\langle A | \mathcal{R} \rangle$ is a nontrivial word $w \in F(A)$, such that there exist distinct $r, r' \in \mathcal{R}$ both beginning with $w$. 
\end{defn}

\begin{defn}\label{def: smallcancellation}
    Let $0 < \lambda < 1$. We say a symmetrized, cyclically reduced presentation $\langle A | \mathcal{R} \rangle$ satisfies the $C'(\lambda)$ \textit{small cancellation condition} if for any piece $w$ and any $r \in \mathcal{R}$ containing $w$, we have 
    $$|w| < \lambda |r|.$$ Here $|\cdot|$ denotes unweighted word length with respect to $A \cup A^{-1}$.
\end{defn}

We are most interested in $C'(\lambda)$ small cancellation when $\lambda \leq \frac{1}{6}$, due to the following result of Greendlinger \cite{greendlinger}:
\begin{thm}[Greendlinger's Lemma]\label{thm: greendlinger}
Let $\langle A | \mathcal{R} \rangle$ be a symmetrized, cyclically reduced presentation satisfying the $C'(\lambda)$ small cancellation condition for some $\lambda \leq \frac{1}{6}$. Then for any nontrivial $v \in \langle\langle \mathcal{R} \rangle \rangle$, there exists a relator $r \in \mathcal{R}$ and a common subword $u$ of $v$ and $r$ such that 
$$|u| > (1-3\lambda)|r|.$$
\end{thm}
In other words, any nontrivial word representing the identity with respect to a presentation satisfying a small cancellation condition must contain a large portion of some relator. Another implication of a $C'(\lambda)$ small cancellation for $\lambda < \frac{1}{6}$ is hyperbolicity, which will be used later in this paper.

\subsection{Random Groups}
Groups come in all shapes and sizes, and it is difficult to name any non-trivial property satisfied by all groups. Models of random groups allow us to make the best of this situation, by formalizing what it means for ``most" groups to have some property. For a detailed reference, we refer the reader to Ollivier's survey \cite{olliviersurvey}.

As mentioned in the introduction, we will consider Gromov's density model of random groups, introduced in \cite{gromovhyperbolic} and expanded on in \cite{gromovrandom}. In this model, we begin with a generating set $A = \{a_{1}, \ldots, a_{m}\}$, $m \geq 2$, and a density parameter $0 \leq d \leq 1$. Considering the free group $F(A)$, for each length $\ell \in \mathbb{N}_{+}$ we select $(2m-1)^{d\ell}$ (rounding down when this is not an integer) cyclically reduced words of length $\ell$ uniformly randomly and with replacement. We take this set of cyclically reduced words to be our set of relators $\mathcal{R}_{\ell}$, and so we obtain a random group
$$G_{\ell} = F(A)/\langle \langle \mathcal{R}_{\ell} \rangle \rangle$$ for each $\ell$. We say a property $P$ of groups is \textit{generic} at density $d$ if the probability that $G_{\ell}$ has property $P$ approaches $1$ as $\ell \to \infty$. 

Some important generic properties include hyperbolicity and $C'(2d+\eps)$ small cancellation (for any $\eps > 0$) at density $d < \frac{1}{2}$ \cite[9.B]{gromovrandom}.

\begin{rmk}\label{rmk: reducedorcyclic}
    As noted by Ollivier in \cite{olliviersurvey}, one may wonder whether there is a difference between randomly selecting cyclically reduced words (as we do) or  reduced words. Similarly, there may be concerns about whether we should select randomly with replacement or not. Tracing through the relevant proofs in Section \ref{sec: randomgroups}, we see that these considerations do not matter, so we make our choices for consistency and convenience.
\end{rmk}

\section{Existence and Rough Uniqueness of Minimizers}\label{sec: existenceanduniqueness}
In this section $G$ will be a torsion-free, non-elementary hyperbolic group, and $S$ a finite, symmetric generating set for $G$. Let $\mathcal{W}$ denote the set of normalized weights on $S$, which is closed under taking convex combinations. Using basic properties of torsion-free hyperbolic groups, we will show the existence of a minimizing weight $w \in \mathcal{W}$ satisfying 
$$h(w) = \inf_{w' \in \mathcal{W}} h(w').$$ Then, using the convexity result of Cantrell-Tanaka \cite{cantrelltanaka}, we establish the uniqueness of this minimizing weight up to rough isometry of the induced weighted word metrics. As a simple application of this result, we show that surface groups (with the standard generators) have a unique normalized minimizing weight, which is in fact the uniform weight.

We will use the following basic properties of the volume entropy $h$:
\begin{itemize}
    \item For $\alpha > 0$, $d_{\alpha w} = \alpha d_{w}$, and $h(\alpha w) = \frac{1}{\alpha}h(w)$.
    \item If $d_{w} \leq d_{w'}$, then $h(w) \geq h(w')$.
    \item Identifying $\mathcal{W}$ with a subset of $\RR^{|S|}$ via the map $w \mapsto (w(s_{1}), w(s_{2}), \ldots, w(s_{|S|}))$, $h$ is a continuous function $\mathcal{W} \to \RR$.
\end{itemize}

\subsection{Existence} 
To establish the existence of minimizers, we show that the volume entropy $h$ approaches infinity as $w_{\min}= \min_{s \in S} w(s)$ approaches zero, hence the space of possible minimizers is compact. The existence of a minimizer then follows from the continuity of $h$. 

The main idea is that sufficiently high powers of any two noncommuting elements in a torsion-free, nonelementary hyperbolic group generate a free subgroup. This is a standard result due to Gromov. 
\begin{lem}[Gromov]\label{lem: freesubgroup}
    Let $G$ be a torsion-free, nonelementary hyperbolic group. There exists an integer $K = K(G)$ so that for any noncommuting elements $x, y \in G$, the subgroup $\langle x^{K}, y^{K} \rangle$ is a free subgroup.
\end{lem}
Therefore, we can bound $h$ from below by the entropy of a weighted $F_{2}$. Then, using the following formula, derived by Balacheff-Merlin \cite{balacheffmerlin} using the calculations in \cite{lim}, we see that $h$ approaches infinity as $w_{\min}$ approaches zero.
\begin{lem}[{\cite[Lemma 5]{balacheffmerlin}}]\label{lem: balacheffmerlin}
    Consider the free group $F_{k}$, $k \geq 2$, with the standard generating set $S = \{a_{1}^{\pm}, \ldots, a_{k}^{\pm}\}$. Given a symmetric weight $w \colon S \to (0, \infty)$ on $S$, the entropy $h = h(w)$ satisfies
    $$\sum_{i=1}^{k}\frac{1}{1+e^{w(a_{i})h}} = \frac{1}{2}.$$
\end{lem}

Putting these results together, we have:
\begin{prop}\label{prop: compactness}
    Let $(G, S)$ be a torsion-free, nonelementary hyperbolic group with a symmetric generating set, and $w \colon S \to (0, \infty)$ a weight on $S$. For each $N > 0$ there exists $\delta > 0$ so that $w_{\min} < \delta$ implies $h(w) > N$.
\end{prop}
\begin{proof}
    Let $K$ be the integer guaranteed by Lemma \ref{lem: freesubgroup}. For a fixed $N > 0$, there exists $\delta' > 0$ so that the free group $(F_{2}, \{a^{\pm}, b^{\pm}\})$ with weight $w'(a^{\pm}) < \delta'$, $w'(b^{\pm}) < K$ has entropy $h(F_{2}, w')> N$. This is straightforward from Lemma \ref{lem: balacheffmerlin}, since 
    $$\frac{1}{1+e^{a N}}  + \frac{1}{1+e^{KN}}> \frac{1}{2}$$ for sufficiently small $a$. Now, let $\delta = \delta'/K$, and suppose that $w_{\min} < \delta$. Let $s \in S$ be a generator with $w(s) = w_{\min}$. Since a torsion-free nonelementary hyperbolic group is centerless, there exists $s' \in S$ such that $s, s'$ do not commute. By Lemma \ref{lem: freesubgroup}, $\langle s^{K}, s'^{K} \rangle$ is a free subgroup of $G$. Therefore the entropy $h(G, w)$ is bounded below by the entropy $h(F_{2}, w')$, where $w'(a) = d_{w}(e, s^{K})$ and $w'(b) = d_{w}(e, s'^{K})$. Noting that $d_{w}(e, s^{K}) \leq Kw(s) < \delta'$, and $d_{w}(e, s'^{K}) \leq Kw(s') < K$, we conclude that $h(G, w) \geq h(F_{2}, w') > N$.
\end{proof}
The existence of minimizers follows.
\begin{prop}\label{prop: existence}
    Let $(G, S)$ be a torsion-free, nonelementary hyperbolic group with a symmetric generating set. There exists a normalized weight $w \in \mathcal{W}$ with 
    $$h(w) = \inf_{w' \in \mathcal{W}}h(w').$$
\end{prop}
\begin{proof}
    Let $w_{*}$ be the uniform weight given by $w_{*}(s) = \frac{1}{|S|}$ for all $s \in S$, and let $H = h(w_{*})$. By Proposition \ref{prop: compactness}, there exists $\delta > 0$ so that $w_{\min} < \delta$ implies $h(w) > H$. Therefore, 
    $$\inf_{w' \in \mathcal{W}}h(w') = \inf_{w' \in \mathcal{W}_{\geq \delta}}h(w'),$$ where 
    $\mathcal{W}_{\geq \delta} = \{w \in \mathcal{W} \colon w_{\min} \geq \delta\}.$ This set is compact, so by continuity of $h$ the infimum is achieved. 
\end{proof}

\begin{rmk}
    Let us note that the torsion-free assumption is necessary for existence. Consider $G = F_{2} \times \ZZ_{2}$, with generating set $S = \{a^{\pm}, b^{\pm}, 1\}$. For any weight $w$ on $S$, it is easy to see that $h(w)$ is equal to the volume entropy of $(F_{2}, \{a^{\pm}, b^{\pm}\})$ with the weight $w|_{F_{2}}$, given by $w|_{F_{2}}(a^{\pm}) = w(a^{\pm})$, $w|_{F_{2}}(b^{\pm}) = w(b^{\pm})$. 
    Therefore, for a normalized weight $w$, $h(w)$ can always be decreased while preserving normalization by decreasing $w(1)$ and scaling up $w(a), w(b)$. It follows that there is no minimizing weight.
\end{rmk}

\subsection{Convexity and Rough Uniqueness}
Now, we show that the map $w \mapsto h(w)$ is convex on $\mathcal{W}$, and in fact is strictly convex up to rough isometry of the induced metric $d_{w}$. We will show that this is true more generally for $d \in \mathcal{D}(G)$, which by Proposition \ref{prop: weightedishyperbolic} contains all $d_{w}$. 

\begin{prop}\label{prop: roughconvexity}
    Let $d_{0}, d_{1} \in \mathcal{D}(G)$ and $t \in [0, 1]$ be arbitrary. Let $d_{t} = (1-t)d_{0} + td_{1} \in \mathcal{D}(G)$. Then we have 
    $$h(d_{t}) \leq (1-t)h(d_{0}) + th(d_{1}).$$ Moreover, if equality holds and $t \neq 0, 1$, then $d_{0}, d_{1}$ are roughly isometric. 
\end{prop}
\begin{proof}
    We know that 
    \begin{align*}
    h(d_{t}) &= \inf\left\{a \in \RR \colon \sum_{x \in G} \exp(-ad_{t}(e, x)) < \infty\right\}\\
    &= \inf\left\{a \in \RR \colon \sum_{x \in G} \exp(-atd_{1}(e, x) - a(1-t)d_{0}(e, x)) < \infty\right\},
    \end{align*} where the first equality is equivalent to the fact that $(h(d_{t}), 0)$ is on the Manhattan curve $\mathcal{C}(d_{t}, d_{t})$.

    Let $\theta_{d_{0}, d_{1}}$ be the function parameterizing the Manhattan curve $\mathcal{C}(d_{0}, d_{1})$. By the above equality, along with the definition of the Manhattan curve, we have 
    $$(1-t)h(d_{t}) = \theta_{d_{0}, d_{1}}(th(d_{t})).$$
    
    We know that $\theta_{d_{0}, d_{1}}$ is convex, and we have $\theta_{d_{0}, d_{1}}(0) = h(d_{0})$, $\theta_{d_{0}, d_{1}}(h(d_{1})) = 0$. Therefore, 
    \begin{align*}
    \theta_{d_{0}, d_{1}}(th(d_{t})) &\leq \left(1 - \frac{th(d_{t})}{h(d_{1})} \right)\theta_{d_{0}, d_{1}}(0) + \frac{th(d_{t})}{h(d_{1})}\theta_{d_{0}, d_{1}}(h(d_{1}))\\
    &= h(d_{0}) - \frac{h(d_{0})}{h(d_{1})}th(d_{t}).
    \end{align*}
    Combining with the previous equality and rearranging, we have
    $$[th(d_{0}) + (1-t)h(d_{1})]h(d_{t}) \leq h(d_{0})h(d_{1}),$$ and thus
    $$h(d_{t}) \leq \frac{h(d_{0})h(d_{1})}{th(d_{0}) + (1-t)h(d_{1})} \leq (1-t)h(d_{0}) + th(d_{1}),$$ proving the desired inequality.

    Here, the last inequality follows from the computation 
    \begin{align*}
&[th(d_{0}) + (1-t)h(d_{1})][(1-t)h(d_{0}) + th(d_{1})] - h(d_{0})h(d_{1}) \\
&= t(1-t)(h(d_{0}) - h(d_{1}))^{2} \geq 0,
\end{align*}
with equality holding only when either $t = 0, 1$ or $h(d_{0}) = h(d_{1})$. Therefore, if $t \neq 0, 1$ and $h(d_{t}) = (1-t)h(d_{0}) + th(d_{1})$, then $h(d_{0}) = h(d_{1})$. Moreover, all the previous inequalities must be equalities, and thus $(th(d_{t}), \theta_{d_{0}, d_{1}}(th(d_{t})))$ lies on the line connecting $(0, h(d_{0}))$, $(h(d_{1}), 0)$. By convexity of the Manhattan curve, the curve must be a straight line between $(0, h(d_{0}))$, $(h(d_{1}), 0)$. By Theorem \ref{thm: manhattanline}, we conclude there exists $\lambda > 0$ so that $d_{0}, \lambda d_{1}$ are roughly isometric. It is clear that roughly isometric metrics have the same volume entropy, hence 
$$h(d_{0}) = h(\lambda d_{1}) = \frac{1}{\lambda}h(d_{1}) = \frac{1}{\lambda}h(d_{0}),$$ and so $\lambda = 1$. Therefore, $d_{0}, d_{1}$ are roughly isometric, which proves the equality case. 
\end{proof}

Now, we discuss convex combinations of weights. 
\begin{lem}\label{lem: convexcombination}
    Let $w_{0}, w_{1} \in \mathcal{W}$, $t \in [0, 1]$ be arbitrary. Let $w_{t} = (1-t)w_{0} + tw_{1}$. Then 
    $$d_{w_{t}} \geq (1-t)d_{w_{0}} + td_{w_{1}}.$$
\end{lem}
\begin{proof}
    Let $x \in G$ be arbitrary, and let $s_{1}, \ldots, s_{n} \in S$ be so that $x = \overline{s_{1} \cdots s_{n}}$ and $d_{w_{t}}(e, x) = \sum_{i=1}^{n}w_{t}(s_{i})$. Then for $j = 0, 1$
    we have $d_{w_{j}}(e, x) \leq \sum_{i=1}^{n}w_{j}(s_{i}).$ Therefore
    \begin{align*}
    d_{w_{t}}(e, x) = \sum_{i=1}^{n}w_{t}(s_{i}) &= (1-t)\sum_{i=1}^{n}w_{0}(s_{i}) + t\sum_{i=1}^{n}w_{1}(s_{i}) \\ &\geq (1-t)d_{w_{0}}(e, x) + td_{w_{1}}(e, x).
    \end{align*}The general claim follows by left-invariance.
\end{proof}

From this we can deduce strict convexity and uniqueness of minimizers, up to rough isometry of the induced metrics.
\begin{cor}\label{cor: roughconvexity}
   Let $w_{0}, w_{1} \in \mathcal{W}$, $t \in [0, 1]$ be arbitrary. Let $w_{t} = (1-t)w_{0} + tw_{1}$. Then 
   $$h(w_{t}) \leq (1-t)h(w_{0}) + th(w_{1}),$$ and if equality holds for $t \neq 0, 1$, the metrics $d_{w_{0}}, d_{w_{1}}$ are roughly isometric.
\end{cor}
\begin{proof}
    We have 
    \begin{align*}
        h(w_{t}) = h(d_{w_{t}}) &\leq h((1-t)d_{w_{0}} + td_{w_{1}})\\
        &\leq (1-t)h(d_{w_{0}}) + th(d_{w_{1}})\\
        &= (1-t)h(w_{0}) + th(w_{1}),
    \end{align*}
    the first inequality following from Lemma \ref{lem: convexcombination} and the second inequality following from Proposition \ref{prop: roughconvexity}. If equality holds for some $t \neq 0, 1$, then the inequalities above are equalities, and by the equality clause in Proposition \ref{prop: roughconvexity} we conclude that $d_{w_{0}}, d_{w_{1}}$ are roughly isometric.
\end{proof}

This brings us to the central claim of this subsection: minimizing weights are unique up to rough isometry of their induced weighted word metrics. This follows from Corollary \ref{cor: roughconvexity} and the fact that $\mathcal{W}$ is closed under convex combinations.
\begin{cor}\label{cor: roughuniqueness}
    Suppose $w_{0}, w_{1} \in \mathcal{W}$ minimize $h$. Then $d_{w_{0}}$ and $d_{w_{1}}$ are roughly isometric.
\end{cor}

\subsection{Surface Groups}

We apply the previous results to show that a surface group, equipped with the standard set of generators, has the uniform weight as its unique normalized minimizer. Consider the surface group $S_{g}$ of genus $g \geq 2$, which is torsion-free and non-elementary hyperbolic. Let $S = \{a_{1}^{\pm}, b_{1}^{\pm}, \ldots, a_{g}^{\pm}, b_{g}^{\pm}\}$ be the standard symmetric generating set for $S_{g}$, and take the standard presentation $$S_{g} = \langle a_{1}, b_{1}, \ldots, a_{g}, b_{g} |  [a_{1}, b_{1}] \cdots [a_{g}, b_{g}] = e\rangle.$$ 

The strategy is as follows. First, we use symmetries of the surface group to show that the uniform weight is a normalized minimizer. Then, by Corollary \ref{cor: roughuniqueness}, it suffices to check that non-uniform weights give rise to metrics which are not roughly isometric to the metric induced by the uniform weight.

We begin by stating the following lemma, showing powers of generators are geodesic words in surface groups. 

\begin{lem}[{\cite[Lemma 2.2]{sambusetti}}]\label{lem: surfacegeodesics}
   Let $x$ be a word on $S$ such that $x$ does not contain any subword of a cyclic permutation of  $r = [a_{1}, b_{1}] \cdots [a_{g}, b_{g}]$ or $r^{-1}$ of length $2g -2$. Then $x$ is a geodesic representative for $\overline{x} \in S_{g}$ with respect to the unweighted word length induced by $S$.
\end{lem}
As a corollary, we conclude that powers $s^{n}$ of generators $s \in S$ are geodesic with respect to the uniform weight on $S$.
\begin{cor}\label{cor: surfacegeodesics}
    For any $s \in S$ and $n \geq 1$, $s^{n}$ is a geodesic word with respect to the uniform weight on $S$.
\end{cor}
\begin{proof}
    Note that being a geodesic word with respect to the uniform weight is equivalent to being a geodesic word with respect to the unweighted word length. Since any cyclic permutation of $r = [a_{1}, b_{1}] \cdots [a_{g}, b_{g}]$ or $r^{-1}$ does not contain $s^{2}$ as a subword for any $s \in S$, and since $2g - 2 \geq 2$ for $g \geq 2$, we are done by Lemma \ref{lem: surfacegeodesics}.
\end{proof}

Now we show that the uniform weight is indeed a normalized minimizer.

\begin{prop}\label{prop: uniformsurface}
    The uniform weight $w_{*}$ on $S$ is a minimizing normalized weight on $(S_{g}, S)$.
\end{prop}
\begin{proof}
    Let $w$ be a minimizing normalized weight on $(S_{g}, S)$, which exists by Proposition \ref{prop: existence}. Consider the automorphisms $\varphi, \psi_{i}$, $1 \leq i \leq g$, of $S_{g}$, induced by 
    $$\begin{cases}\psi_{i}(a_{j}) &= a_{j+i} \\ \psi_{i}(b_{j}) &= b_{j+i}\end{cases} \; \text{and} \; \begin{cases} \varphi(a_{j}) &= b_{g+1-j} \\ \varphi(b_{j}) &= a_{g+1-j},\end{cases}$$ where the indices are taken modulo $g$ (replacing $0$ with $g$). Define  $w_{\psi_{i}} = w \circ \psi_{i}$ and $w_{\varphi_{i}} = w \circ \psi_{i} \circ \varphi$, using the fact that $\varphi, \psi_{i}$ induce bijections $S \to S$. We see that for any $x \in S_{g}$, 
    $d_{w}(e, \psi_{i}(x)) = d_{w_{\psi_{i}}}(e, x),$ and so $h(w_{\psi_{i}}) = h(w)$ for each $i$. Similarly, $h(w_{\varphi_{i}}) = h(w)$ for each $i$. So, $w_{\psi_{i}}, w_{\varphi_{i}}$ are all minimizing normalized weights. 
    
    By Corollary \ref{cor: roughconvexity} any convex combination of minimizing normalized weights is once again a minimizing normalized weight, so we see that 
    $$\overline{w} = \frac{1}{|2g|}\sum_{i=1}^{g}(w_{\psi_{i}} + w_{\varphi_{i}})$$ must be a minimizing normalized weight as well. Computing, for each $j$ we have 
    \begin{align*}
    \overline{w}(a_{j}) &= \frac{1}{2g}\sum_{i=1}^{g}(w_{\psi_{i}}(a_{j}) + w_{\varphi_{i}}(a_{j})) \\
    &= \frac{1}{2g} \left(\sum_{i=1}^{g}w(\psi_{i}(a_{j})) + \sum_{i=1}^{g}w(\psi_{i}(\varphi(a_{j})))\right)\\
    &= \frac{1}{2g}\left(\sum_{i=1}^{g}w(a_{i}) + \sum_{i=1}^{g}w(b_{i})\right) = \frac{1}{4g},
    \end{align*}
    where the last equality follows because $w$ is normalized (we haven't accounted for inverses, hence the $\frac{1}{2}$ factor). Likewise, we see that $\overline{w}(b_{j}) = \frac{1}{4g}$ for all $j$, hence $\overline{w}$ is the uniform weight $w_{*}$ on $S$.
\end{proof}

Finally we have our desired result. 
\begin{thm}\label{thm: uniquesurface}
   The uniform weight $w_{*}$ on $S$ is the unique normalized weight on $(S_{g}, S)$ minimizing $h$.
\end{thm}
\begin{proof}
Suppose, for a contradiction, that $w$ is a normalized minimizing weight on $S$ with $w \neq w_{*}$. Then there exists some $s \in S$ with $w(s) < w_{*}(s)$. By Corollary \ref{cor: surfacegeodesics} we know that $s^{n}$ is a geodesic word with respect to $w_{*}$ for all $n$, so 
$$\ell_{d_{w_{*}}}(s) = w_{*}(s) > w(s) \geq \ell_{d_{w}}(s).$$ However, since $w_{*}, w$ are both normalized minimizers, by Corollary \ref{cor: roughuniqueness} the metrics $d_{w_{*}}$ and $d_{w}$ are roughly isometric, contradicting Theorem \ref{thm: translationisometric}.

\end{proof}
\begin{rmk}
    Note that the arguments in this section can also be applied to show that the uniform weight on the free group $F_{m}$ with standard generating $S = \{a_{1}^{\pm}, \ldots, a_{m}^{\pm}\}$ is a minimizer. The uniqueness of this minimizer is similarly clear, so this recovers a special case of Lim's results in \cite{lim}, viewing the Cayley graph of $(F_{m}, S)$ as the universal cover of a bouquet of $m$ circles. 
\end{rmk}

\section{Translation-Apparent Presentations and Strict Uniqueness} \label{sec: translationapparent}
As shown in the previous section, every torsion-free, non-elementary hyperbolic group $(G, S)$ admits a unique minimizing normalized weight up to rough isometry. When can we improve this to actual uniqueness, and avoid passing to a rough isometry class? 

We define a condition on a group presentation, that when satisfied, guarantees uniqueness. In fact, this condition will actually imply strict convexity, which in turn implies the uniqueness of the minimizer.  The key idea is that roughly isometric metrics have the same translation lengths of elements, hence we need a condition which guarantees distinct weights have distinct translation lengths. 

\subsection{Uniqueness} Given an alphabet $A$, $a \in A$, and $v \in F(A)$, let $\#_{a}(v)$, $\#_{a^{-1}}(v)$ denote the total number of occurrences of $a$ and $a^{-1}$ in $v$, respectively. Let $\#_{a^{\pm}}(v) = \#_{a}(v) + \#_{a^{-1}}(v)$. We denote word length with respect to $A \cup A^{-1}$ by $|\cdot|$.
\begin{defn}\label{def: translationapparentpresentation}
    We say a presentation $G = \langle A | \mathcal{R} \rangle$  is a $\lambda$-\textit{translation-apparent} presentation  for the marked group $(G, S)$ if the following conditions are satisfied:
    \begin{enumerate}
        \item $S = A \cup A^{-1},$
        \item (Small Cancellation) The presentation $\langle A | \mathcal{R} \rangle$ is symmetrized, cyclically reduced, and satisfies the $C'(\lambda)$ small cancellation condition,
        \item (Even Distribution) Every $r \in \mathcal{R}$ satisfies the following conditions: 
        \begin{enumerate}[(i)]
            \item For any $s \in S$, if $u$ is a subword of $r$ of the form $s^{n}$, then $n < \lambda|r|$. 
            \item For any $a \in A$, $$\#_{a^{\pm}}(u) < \frac{1}{2}\#_{a^{\pm}}(r)$$ for any subword $u$ of $r$ with $|u| = \lceil 4\lambda|r|\rceil$.
            \item For any $a \in A$,  $$\#_{a^{\pm}}(u) > \frac{1}{8m} |u|$$ for any subword $u$ of $r$ with $|u| = \lceil \lambda|r| \rceil $. Here $m = |A|$.
        \end{enumerate}
    \end{enumerate}
\end{defn}

To obtain useful results from the second even distribution condition, we will need $\lambda \leq \frac{1}{8}$, which will also come in handy when applying Greendlinger's lemma. In practice, we will take $\lambda = \frac{1}{16}$. The main consequence of the existence of a translation-apparent presentation is, as the name suggests, the translation length of a generator $a$ is equal to $w(a)$ for any weight $w$. 

To elaborate, the first and second even distribution conditions together will allow us to show that words of the form $s^{n}$, for $s \in S$, are $w$-geodesics for any weight $w$. Here, the first condition complements the use of Greendlinger's lemma, in order to show that any other representative for $\overline{s^{n}}$ must contain a large (in the unweighted sense) part of some relator. Then the second even distribution condition guarantees that even if one of the weights $w(s)$ is extremely large compared to the weights of other generators, it cannot skew all the weight to a small subword of a relator. From this it follows that any other representative has large weight, and thus fails to be geodesic. 

In fact, we can prove that a broader class of words, defined below, are $w$-geodesic for any $w$. Note that the third even distribution condition will not be used in this proof, but will be important for the counting arguments in the next subsection.

The following definition is given in \cite{shukhov}.  
\begin{defn}\label{def: reducedgeodesic}
    Let $G = \langle A | \mathcal{R}\rangle$ be a presentation for the marked group $(G, S)$. We say a word $x \in F(A)$ is $\alpha$-reduced if for any $r \in \mathcal{R}$, any common subword $u$ of $x$ and $r$ satisfies
    $|u| < \lceil\alpha |r|\rceil$. We denote the set of such words by $\mathcal{N}(\alpha)$.
\end{defn}
We can now generalize the argument in \cite[Lemma 1]{shukhov} to the weighted case.
\begin{prop}\label{prop: stronggeodesics}
 Let $G = \langle A | \mathcal{R}\rangle$ be a $\lambda$- translation-apparent presentation for the marked group $(G, S)$. Let $w$ be an arbitrary weight on $S$ (not necessarily normalized). Then any $x \in \mathcal{N}(\lambda)$ is the unique geodesic representative of $\overline{x} \in G$ with respect to $w$. In particular, any distinct $x, y \in \mathcal{N}(\lambda)$ satisfy $\overline{x} \neq \overline{y} \in G$.
\end{prop}
\begin{proof}
    Suppose, for a contradiction, that there exists $y \in F(A)$ such that $x \neq y$, $\overline{x} = \overline{y}$, and $|x|_{w} > |y|_{w}$. Without loss of generality, we may assume that $y$ has shortest $w$-length among all such words.
    
    Write $x = x_{1}v$, $y = y_{1}v$ (where $v$ may be the empty word), so that $x_{1}y_{1}^{-1}$ is reduced. Then $\overline{x_{1}y_{1}^{-1}} = \overline{xy^{-1}} = e$, hence by Greendlinger's lemma there exists $r \in \mathcal{R}$ and a common subword $u$ of $x_{1}y_{1}^{-1}$ and $r$ satisfying 
    $|u| > (1-3\lambda)|r|.$
    
    Now write $u$ as a reduced concatenation $u_{1}u_{2}$, where $u_{1}$ is a subword of $x_{1}$ and $u_{2}$ is a subword of $y_{1}^{-1}$. Since $x \in \mathcal{N}(\lambda)$ and $u_{1}$ is a common subword of $x$ and $r$, we have $|u_{1}| \leq \lambda|r|$. It follows that 
    $$|u_{2}|  = |u| - |u_{1}| > (1-4\lambda)|r|.$$

    Let $u_{3}$ be the unique word so that $u_{2}u_{3}$ is reduced, and is a permutation of $r$. We have $|u_{3}| \leq 4\lambda|r|$. By the second even distribution condition for translation-apparent presentations, we conclude that 
    $$\#_{a^{\pm}}(u_{3}) < \frac{1}{2}\#_{a^{\pm}}(r)$$ for any $a \in A$. Since $\#_{a^{\pm}}(u_{2}) + \#_{a^{\pm}}(u_{3}) = \#_{a^{\pm}}(r)$, we conclude that 
    $$\#_{a^{\pm}}(u_{2}) > \#_{a^{\pm}}(u_{3})$$ for each $a \in A$. Therefore $|u_{2}|_{w} > |u_{3}|_{w} = |u_{3}^{-1}|_{w}$. Replacing the subword $u_{2}^{-1}$ of $y$ with $u_{3}$ (and reducing if necessary), we obtain a new word $y' \in F(A)$ with $\overline{y'} = \overline{y}$ and $|y'|_{w} < |y|_{w}$, contradiction.
\end{proof}
We immediately have the following corollary.
\begin{cor}\label{cor: translationapparentlengths}
    Let $G = \langle A | \mathcal{R}\rangle$ be a $\lambda$-translation-apparent presentation for the marked group $(G, S)$. Then for any weight $w$ on $S$ and any $s \in S$, we have $\ell_{d_{w}}(s) = w(s).$ 
\end{cor}
\begin{proof}
   Let $w$ be an arbitrary weight. By the first even distribution condition, any word of the form $s^{n}$, $s \in S$, is in $\mathcal{N}(\lambda)$. By Proposition \ref{prop: stronggeodesics}, $s^{n}$ is a $w$-geodesic representative for $\overline{s^{n}}$, and so $$d_{w}(e, \overline{s^{n}}) = w(s^{n}) = nw(s).$$ The claim follows from the definition of $\ell_{d_{w}}$.
\end{proof}

This in turn implies torsion-free, hyperbolic groups with translation-apparent presentations have unique normalized minimizers, which is Theorem \ref{thm: mainthm1} in the introduction.

\begin{thm}[Theorem \ref{thm: mainthm1}]\label{thm: translationapparentuniqueness}
    Let $(G, S)$ be a torsion-free, non-elementary hyperbolic group admitting a $\lambda$-translation-apparent presentation $\langle A | \mathcal{R} \rangle$. Then the volume entropy is a strictly convex function on the set of normalized weights, and in particular there exists a unique normalized weight minimizing the volume entropy.
\end{thm}
\begin{proof}
    Corollary \ref{cor: roughconvexity} and Corollary \ref{cor: roughuniqueness} respectively tell us that the volume entropy is strictly convex up to rough isometry, and that there is a unique minimizing normalized weight up to rough isometry. So it suffices to show that any distinct normalized weights $w_{0}, w_{1}$ on $S$ induce metrics $d_{w_{0}}, d_{w_{1}}$ which are not roughly isometric.

    Since $w_{0} \neq w_{1}$, there exists $s \in S$ such that $w_{0}(s) \neq w_{1}(s)$. By Corollary \ref{cor: translationapparentlengths}, we have 
    $$\ell_{d_{w_{0}}}(s) = w_{0}(s) \neq w_{1}(s) = \ell_{d_{w_{1}}}(s).$$ By Theorem \ref{thm: translationisometric}, $d_{w_{0}}$ and $d_{w_{1}}$ cannot be roughly isometric.
\end{proof}

\subsection{Bounding Entropy} \label{sec: counting}

In this subsection we give lower bounds for the entropy of $(G, S, w)$ when $(G, S)$ admits a $\lambda$-translation-apparent presentation. The key ingredient is a system of generating functions developed in \cite{myers}, which we use to count weighted $\lambda$-reduced words. Unlike before, we will initially assume that our fixed weight $w$ takes integer values.

Let $G = \langle A | \mathcal{R}\rangle$ be a $\lambda$-translation-apparent presentation for $(G, S)$, and let $\mathcal{N}(\lambda)$ denote the set of $\lambda$-reduced words. Let $g(n)$ denote the number of words in $\mathcal{N}(\lambda)$ with total weight $ \leq n$. By Proposition \ref{prop: stronggeodesics}, the evaluation map $\mathcal{N}(\lambda) \to G$ is injective, and moreover $|x|_{w} = d(e, \overline{x})$ for any $x \in \mathcal{N}(\lambda)$. Therefore, $g(n) \leq |B_{d_{w}}(n)|$, and 
$$\lim_{n \to \infty} \frac{\ln g(n)}{n} \leq h(G, w).$$ 

The existence of the limit $\lim_{n \to \infty} \frac{\ln g(n)}{n}$, as well as the following lemma, follow from the main result of \cite{myers}. For the sake of readability, we leave the proof of the next lemma to the appendix. 

\begin{lem}\label{lem: generatinginequality}
    Let $M = \lim_{n \to \infty} \sqrt[n]{g(n)}$, so that $\ln M = \lim_{n \to \infty} \frac{\ln g(n)}{n}$. Let $$\mathcal{R}_{\lambda, s} = \{u |  u \text{ is a subword of } r \in \mathcal{R} \text{ ending with } s, |u| = \lceil \lambda |r| \rceil\}.$$  Then any real root of
    $$p(z) \coloneqq \left( 1 - \sum_{s \in S}(z^{w(s)}+1)^{-1}\right) + \left(\sum_{s \in S} \frac{z^{w(s)}}{z^{w(s)}+1}\sum_{u \in \mathcal{R}_{\lambda, s}}\frac{1}{z^{|u|_{w}}}\right)$$ 
    is $\leq M$.
\end{lem}
\begin{proof}
    See Appendix \ref{sec: appendix}.
\end{proof}

Now consider the free group $F_{m} = F(A)$, where $m = |A|$. We can equip $(F_{m}, S)$ with the same weight $w$ as $(G, S)$. 
\begin{lem}\label{lem: freegroupgenerating}
   Define $p_{0}(z) \coloneqq 1 - \sum_{s \in S}(z^{w(s)}+1)^{-1}$. Then $M_{0} \coloneqq e^{h(F_{m}, w)}$ is a root of $p_{0}$.
\end{lem}
\begin{proof}
        By Lemma \ref{lem: balacheffmerlin}, $h(F_{m}, w)$ satisfies 
    $$1 - \sum_{i=1}^{m}2(e^{w(a_{i})h(F_{m}, w)}+1)^{-1} = 0.$$ This shows that $M_{0}$ is a root of $p_{0}(z)$.
\end{proof}

Putting the above lemmas together, we derive the desired lower bound on $M$. Note that the quantities $m, j, \ell$ are independent of the weight $w$.
\begin{prop}\label{prop: integerreducedinequality}
    Let $j = \max_{s \in S} \left|\mathcal{R}_{\lambda, s}\right|$ and let $l$ denote the shortest unweighted length of a word in $\bigcup_{s \in S} \mathcal{R}_{\lambda, s}$.Suppose that $l > 32m$ and 
    $$8mjl < (2m-1)^{\frac{l}{16m} -2}.$$ Then 
     $$M_{0} - \frac{1}{Nl} <   M \leq M_{0},$$ where $N = \sum_{s \in S} w(s)$.
\end{prop}
\begin{proof}
  The inequality $M \leq M_{0}$ is clear. For the other direction, we show that $p$ has a real root in the interval $\left(M_{0} - \frac{1}{Nl}, M_{0}\right)$, so that the desired inequality follows from Lemma \ref{lem: generatinginequality}.

  First, we establish the inequality $M_{0} - \frac{1}{Nl} > \frac{M_{0}+1}{2} > 1$. Note that 
  \begin{align*}
  M_{0}^{N} = e^{Nh(F_{m}, w)} = e^{h(F_{m},\frac{1}{N}w)} &\geq e^{h(F_{m}, w_{*})} \\ &= 2m-1 > e,
  \end{align*}where $w_{*}$ is the uniform normalized weight on $(F_{m}, S)$. The first inequality above follows from \cite{lim}, which states that the uniform normalized weight is the unique normalized minimizing weight on $(F_{m}, S)$. 
  It follows that
  $M_{0} -1 > e^{\frac{1}{N}} -1 \geq \frac{1}{N} ,$ 
  and so 
  $$M_{0} - \frac{M_{0}+1}{2} = \frac{M_{0}-1}{2} > \frac{1}{2N} \geq \frac{1}{Nl}$$
  as claimed.

  By the Mean Value Theorem, there exists $a \in \left[M_{0} - \frac{1}{Nl}, M_{0}\right]$ such that 
    $$-p_{0}\left(M_{0} - \frac{1}{Nl}\right) = p_{0}(M_{0}) - p_{0}\left(M_{0} - \frac{1}{Nl}\right)  = \frac{p'_{0}(a)}{Nl}.$$ Here we've used the fact that $p_{0}(M_{0}) = 0$ by Lemma \ref{lem: freegroupgenerating}. Since $a \geq M_{0} - \frac{1}{Nl} > 1$ and $1 \leq w(s) \leq N$ for each $s$, we have
    \begin{align*}
        p'_{0}(a) = \sum_{s \in S}\frac{w(s)a^{w(s)-1}}{(a^{w(s)}+1)^{2}} \geq \sum_{s \in S} \frac{w(s)}{(M_{0}^{N}+1)^{2}} = \frac{N}{(M_{0}^{N}+1)^{2}} \geq \frac{N}{4M_{0}^{2N}}, 
    \end{align*}
    and so 
    $$p_{0}\left(M_{0} - \frac{1}{Nl}\right) < -\frac{1}{4l M_{0}^{2N}}.$$

 Now define
    $$q(z) \coloneqq p(z) - p_{0}(z) = \sum_{s \in S} \frac{z^{w(s)}}{z^{w(s)}+1}\sum_{u \in \mathcal{R}_{\lambda, s}}\frac{1}{z^{|u|_{w}}}.$$ By the third even distribution condition, for any $s \in S$, each $u \in \mathcal{R}_{\lambda, s}$ satisfies $\#_{s}(u) > \frac{1}{8m}|u|$, and so 
    $$|u|_{w} > \sum_{s \in S}\frac{1}{8m}|u|w(s) \geq \frac{Nl}{8m}.$$ For $b > 1$, it follows that 
    \begin{align*}
        q(b) \leq \sum_{s \in S} \sum_{u \in \mathcal{R}_{\lambda, s}}\frac{1}{b^{|u|_{w}}} &\leq \sum_{s \in S} \sum_{u \in \mathcal{R}_{\lambda, s}}\frac{1}{b^{Nl / 8m}} \leq \frac{2mj}{b^{N\ell /8m}}.
    \end{align*}

    Taking $b = M_{0} - \frac{1}{Nl} > 1$, we have 
$$q(b) \leq \frac{2mj}{\left(M_{0} - \frac{1}{Nl}\right)^{Nl/8m}}.$$ We have
$\left(M_{0} - \frac{1}{Nl}\right)^{2}> \left(\frac{M_{0}+1}{2}\right)^{2} > M_{0},$ and so
\begin{align*}
    p(b) = p_{0}(b) + q(b) &< -\frac{1}{4l M_{0}^{2N}} + \frac{2mj}{M_{0}^{Nl/16m}} = \frac{8mjl - M_{0}^{N(\frac{l}{16m}-2)}}{4l M_{0}^{Nl/16m}}.
\end{align*}
Finally, by assumption we have $8mjl < (2m-1)^{\frac{l}{16m} -2} < M_{0}^{N\left(\frac{l}{16m} -2\right)},$ hence $p(b) < 0.$ On the other hand, we see that 
    $$p(M_{0}) = p_{0}(M_{0}) + q(M_{0}) = q(M_{0}) > 0.$$ We conclude that $p$ has a real root in $\left(M_{0} - \frac{1}{Nl}, M_{0}\right)$.
\end{proof}

Finally, we adapt the above proposition to the case of normalized weights.

\begin{prop}\label{prop: normalizedreducedinequality}
Let $G = \langle A | \mathcal{R}\rangle$ be a $\lambda$-translation-apparent presentation. Let $m, j, l$ be defined as above. Suppose that $l > 32m$ and $8mjl < (2m-1)^{\frac{l}{16m}-2}$. Then for any normalized weight $w$ we have 
$$h(F_{m}, w) - \frac{2}{l} \leq h(G, w) \leq h(F_{m}, w).$$
 \end{prop}
 \begin{proof}
     Once again the inequality $h(G, w) \leq h(F_{m}, w)$ is clear. To derive the other inequality, fix an arbitrary $\eps > 0$, and choose a normalized weight $w'$ so that $w'(s) \in \mathbb{Q}$ for each $s$ and $|h(G, w) - h(G, w')|, |h(F_{m}, w) - h(F_{m}, w')| < \frac{\eps}{2}.$ This is possible by continuity of $h(G, -)$ and $h(F_{m}, -)$.

     Now, let $N \in \mathbb{N}$ be so that $Nw'(s) \in \mathbb{N}$ for each $s$, and let $w'' = Nw'$. Set $M = \lim_{n \to \infty}\sqrt[n]{g(n)}$, where $g(n)$ is the number of $\lambda$-reduced words of $w''$-length $\leq n$, and $M_{0} = e^{h(F_{m}, w'')}$. Since $w''$ is an integer-valued weight, we may apply Proposition \ref{prop: integerreducedinequality} to conclude that $M_{0} - \frac{1}{Nl} < M$. It follows that 
     $$\left(1- \frac{1}{Nl}\right)e^{h(F_{m}, w'')} \leq e^{h(F_{m}, w'')} - \frac{1}{Nl} < M \leq e^{h(G, w'')},$$ so taking natural logs gives us 
     $$h(F_{m}, w'') + \ln \left(1- \frac{1}{Nl}\right) < h(G, w'').$$ Multiplying both sides by $N$ and using the fact that $w'' = Nw'$, we have 
     $$h(F_{m}, w') + N\ln \left(1- \frac{1}{N\ell}\right) < h(G, w').$$ Using the inequality $\ln(1-x) \geq -\frac{x}{1-x}$ for $x < 1$, we conclude that 
     $$h(F_{m}, w')  - \frac{2}{l} \leq  h(F_{m}, w)  - N \frac{1/Nl}{1 - 1/Nl} < h(G, w').$$ Finally, by our choice of $w'$, we conclude that 
     $$h(F_{m}, w) - \eps - \frac{2}{l} <  h(G, w).$$ Since $\eps > 0$ was arbitrary, we have our desired result.
 \end{proof}

\section{Random Groups} \label{sec: randomgroups}
In this section we prove our main theorems \ref{thm: mainthm1},  \ref{thm: mainthm2},  \ref{thm: mainthm3}. To do so, we will show that a generic random group on $m$ letters under a certain density (dependent on $m$) satisfies the $\lambda$-translation-apparent condition for some $\lambda$, at which point the desired statements follow from the results of the previous section.

All the conditions, except for the even distribution condition for a translation-apparent presentation, are standard facts about random groups. To address the even distribution condition, we will use a Chernoff-type bound for Markov chains.

Fix $m \geq 2$. Consider the Markov chain $(X_{i})_{i \geq 0}$ with states $$\{a_{1}, \ldots, a_{m}, a_{1}^{-1}, \ldots, a_{m}^{-1}\},$$ initial distribution $\pi = \left(\frac{1}{2m}, \ldots, \frac{1}{2m}\right)$, and transition matrix $M = (M_{xy})$, where 
$$M_{xy} = \begin{cases} 0, \;& y = x^{-1}\\ \frac{1}{2m-1}, \; &\text{otherwise.}\end{cases}$$ Note that this Markov chain is irreducible, reversible, and has stationary distribution equal to the initial distribution. Moreover, we see that the distribution of the $\ell$-step Markov chain $(X_{0}, \ldots, X_{\ell-1})$ corresponds exactly to the uniform distribution on reduced words of length $\ell$ in $F(A)$ under the natural identification.

The following Chernoff-type bound is essential for our argument. We modify the statement in \cite{lezaud} slightly to better suit our purposes.
\begin{thm}[{\cite[Theorem 1.1, Remark 3]{lezaud}}]\label{thm: chernoffbound}
    Let $M$ be an irreducible and reversible Markov chain on a finite set $V$ with stationary distribution $\pi$. Let $f \colon V \to \RR$ satisfy $\sum_{x\in V}\pi(x)f(x) = c$ and $\norm{f-c}_{\infty} \leq 1$. Then with $\pi$ as the initial distribution, any integer $n > 0$, and any real $0 \leq \delta \leq 1$, we have 
    $$\mathbb{P}\left[\left|c - \frac{1}{n}\sum_{j=1}^{n} f(X_{i})\right| \geq \delta \right]\leq 2e^{-\frac{\eps(M)\delta^{2}n}{12}+1}.$$ Here
    $\eps(M) = 1 - \beta_{1}(M),$ where $\beta_{1}(M)$ is the second largest eigenvalue of $M$.
\end{thm}

We will apply this bound for the specific Markov chain defined above, for the functions $f_{i^{\pm}} \colon \{a_{1}, \ldots, a_{m}, a_{1}^{-1}, \ldots, a_{m}^{-1}\} \to \{0, 1\}$ which are the respective indicator functions of the subsets $\{a_{i}, a_{i}^{-1}\}$. Note that in this case, regardless of the value of $i$, the constant $c$ in the above theory is $\frac{1}{m}$. To obtain a more effective upper bound, we will estimate $\eps(M)$ in our specific situation by applying Cheeger's inequality.

\begin{defn}\label{def: cheegerconstant}
    Let $M$ be an irreducible and reversible Markov chain on a finite set $V$ with stationary distribution $\pi$. The \textit{Cheeger constant} of $M$ is defined by  
    $$\Phi(M) = \min_{S \subseteq V, \pi(S) \leq \frac{1}{2}}\frac{Q(S \times S^{c})}{\pi(S)}.$$ Here
    $$Q(S \times S^{c}) = \sum_{x \in S, y \in S^{c}}\pi(x)K(x, y),$$ where $K(x, y)$ is the transition probability.
\end{defn}

\begin{thm}[{\cite[Lemma 3.3]{sinclairjerrum}}]\label{thm: cheegerinequality}
   Let $M$ be an irreducible and reversible Markov chain on a finite set $V$ with stationary distribution $\pi$. The second largest eigenvalue $\beta_{1}(M)$ satisfies 
   $$\beta_{1}(M) \leq 1 - \frac{\Phi(M)^{2}}{2}.$$
\end{thm}

We know that the transition probabilities for the Markov chain above are given by $K(x, y) = 0$ if $y = x^{-1}$, and $K(x, y) = \frac{1}{2m-1}$ otherwise. Therefore, for a (nonempty, proper) subset $S \subset \{a_{1}, \ldots, a_{m}, a_{1}^{-1}, \ldots, a_{m}^{-1}\}$ and any $x \in S$, at least $|S^{c}| -1$ elements $y$ of $S^{c}$ satisfy $K(x, y) = \frac{1}{2m-1}$. Therefore, we have 
$$\frac{Q(S \times S^{c})}{\pi(S)} \geq \frac{\frac{|S^{c}|-1}{2m-1}\sum_{x \in S}\pi(x)}{\pi(S)} = \frac{|S^{c}|-1}{2m-1}.$$ Since the initial distribution $\pi$ assigns $\frac{1}{2m}$ to each state in $\{a_{1}, \ldots, a_{m}, a_{1}^{-1}, \ldots, a_{m}^{-1}\}$, any subset $S$ with $\pi(S) \leq \frac{1}{2}$ satisfies $|S|\leq m$, and so $|S^{c}| \geq m$. We conclude that 
$$\Phi(M) = \min_{S \subseteq V, \pi(S) \leq \frac{1}{2}}\frac{Q(S \times S^{c})}{\pi(S)} \geq \frac{m-1}{2m-1}.$$ Therefore, by Theorem \ref{thm: cheegerinequality}, we have $\eps(M) = 1-\beta_{1}(M) \geq \frac{(m-1)^{2}}{2(2m-1)^{2}}$.

\begin{lem}\label{lem: reducedwordprobability}
    Let $r$ be a uniformly chosen reduced word of length $\ell$. For ease of notation, write $C_{m} = \frac{(m-1)^{2}}{384m^{2}(2m-1)^{2}}.$ The probability that $$ \frac{3}{4m}|r| < \#_{a_{i}^{\pm}}(r) < \frac{5}{4m}|r|$$ for all $i$ is bounded below by 
    $$1 - 2me^{-C_{m}\ell+1}.$$
\end{lem}
\begin{proof}
    Recall that we can identify this uniformly chosen reduced word of length $\ell$ with the $\ell$-step Markov chain $(X_{0}, \ldots, X_{\ell-1})$ defined above. For each $i$, let $f_{i^{\pm}} \colon \{a_{1}, \ldots, a_{m}, a_{1}^{-1}, \ldots, a_{m}^{-1}\}\to \{0, 1\}$ be the indicator function of the subset $\{a_{i}, a_{i}^{-1}\}$. We see that in our identification, the quantity 
    $\frac{1}{\ell} \sum_{j=1}^{\ell}f_{i^{\pm}}(X_{j})$ is equal to $\frac{1}{\ell}\#_{a_{i}^{\pm}}(r).$
    
    As discussed above, the constant $c$ in Theorem \ref{thm: chernoffbound} is $\frac{1}{m}$. So, applying this theorem with the above identification, we conclude that for $0 < \delta < 1$, $$\mathbb{P}\left[\left|\frac{1}{m} - \frac{1}{\ell}\#_{a_{i}^{\pm}}(r)\right| \geq \delta \right]\leq 2e^{-\frac{(m-1)^{2}\delta^{2}}{24(2m-1)^{2}}\ell +1},$$ where we have used the inequality $\eps(M) \geq \frac{(m-1)^{2}}{2(2m-1)^{2}}$ derived above. Since $|r| = \ell$, choosing $\delta = \frac{1}{4m}$ we see that the probability that 
     $$\frac{3}{4m}|r| < \#_{a_{i}^{\pm}}(r) < \frac{5}{4m}|r|$$ is bounded below by $1 - 2e^{-\frac{(m-1)^{2}}{384m^{2}(2m-1)^{2}}\ell+1}$.

     It follows that the probability that the desired inequalities hold for all $1 \leq i \leq m$ is bounded below by 
     $$1 - 2me^{-\frac{(m-1)^{2}}{384m^{2}(2m-1)^{2}}\ell+1}.$$
\end{proof}

\subsection{Probability of Even Distribution}

In this subsection we show that for $\lambda = \frac{1}{16}$, a uniformly chosen reduced word of length $\ell$ on $m$ letters satisfies the even distribution conditions with probability approaching $1$ as $\ell \to \infty$. First, let us introduce the following notation. Suppose there exists $\alpha > 0$ and a subexponential function $f(\ell)$ so that $P \colon \mathbb{N} \to (0, 1)$ satisfies $P(\ell) \leq f(\ell)e^{-\alpha \ell}$ for all $\ell$. Then we write $P \lesssim \exp(-\alpha\ell)$. In practice, $P(\ell)$ will be the probability that a uniformly chosen reduced word of length $\ell$ satisfies some property of words.  

\begin{rmk}\label{rmk: expbounds}
   Suppose $P_{1}(\ell), P_{2}(\ell), \ldots, P_{n}(\ell)$ are the probabilities that properties $Q_{1}, \ldots, Q_{n}$ hold for a uniformly chosen reduced word of length $\ell$, and that each $P_{i} \lesssim \exp(-\alpha_{i} \ell)$. Then, if $P$ is the probability that all the properties $Q_{i}$ simultaneously hold, we see that $P\lesssim \exp(-\min_{i} \alpha_{i} \ell)$.

\end{rmk}

\begin{lem}\label{lem: reducedwordevendistribution1}
    Let $r$ be a random reduced word of length $\ell$. The probability that some subword $u$ of $r$ with $|u| = \lceil \frac{\ell}{16}  \rceil$ is of the form $s^{\lceil \ell/16 \rceil}$ for some $s \in S$ is $\lesssim \exp(-\frac{1}{16}\ln (2m-1) \ell)$.
\end{lem}
\begin{proof}
    Let $u_{j}$ be the subword of $r$ of length $ \lceil \frac{\ell}{16} \rceil$ starting at the $j$th letter of $r$. The distribution of $r$ is the same as the distribution of a $\ell$-step Markov chain, and so the distribution of $u_{j}$ is the uniform distribution on reduced words of length $ \lceil \frac{\ell}{16}\rceil$. In particular, the probability that $u_{j}$ is of the form $s^{ \lceil \frac{\ell}{16} \rceil}$ is $$ \frac{2m}{2m(2m-1)^{ \lceil \frac{\ell}{16}\rceil-1}} \leq \frac{2m}{(2m-1)^{ \lceil \frac{\ell}{16}\rceil}}.$$ 

    Now, there are $\leq \ell$ subwords $u_{j}$ of $r$ of length $ \lceil \frac{\ell}{16} \rceil$. By the union bound, the probability that at least one of these subwords satisfies the desired condition is bounded above by
    $$\frac{2m\ell}{(2m-1)^{ \lceil \frac{\ell}{16}  \rceil}},$$ which is $\lesssim \exp(-\frac{1}{16} \ln (2m-1)\ell)$.
\end{proof}

\begin{lem}\label{lem: reducedwordevendistribution2}
    Let $r$ be a reduced word of length $\ell$. The probability that some subword $u$ of $r$ with $|u| = \lceil \frac{\ell}{4}  \rceil$ satisfies 
    $$\#_{a_{i}^{\pm}}(u) \geq \frac{1}{2}\#_{a_{i}^{\pm}}(r)$$ for some $i$ is $\lesssim \exp(-\frac{C_{m}}{4}\ell)$. Here $C_{m}$ is the same constant defined in Lemma \ref{lem: reducedwordprobability}.
\end{lem}
\begin{proof}
    Without loss of generality we can assume that $\ell$ is large enough so that $\frac{3}{8m} \ell \geq \frac{5}{4m} \lceil \frac{\ell}{4} \rceil$. Let $u_{j}$ be the subword of $r$ of length $ \lceil \frac{1}{4} \ell \rceil$ starting at the $j$th letter of $r$. As in the previous lemma, the distribution of $u_{j}$ is the uniform distribution on reduced words of length $\lceil \frac{\ell}{16} \rceil.$
    
    By Lemma \ref{lem: reducedwordprobability}, the probability that $\#_{a_{i}^{\pm}}(u_{j}) < \frac{5}{4m}|u_{j}|$ for all $i$ is bounded below by 
    $$1 - 2me^{-C_{m}\lceil \frac{\ell}{4}  \rceil+1}.$$ Likewise, the probability that $\#_{a_{i}^{\pm}}(r) > \frac{3}{4m}|r|$ for all $i$ is bounded below by $$1 - 2me^{-C_{m}\ell +1}.$$ Therefore, the probability that all the inequalities
    $$\frac{1}{2}\#_{a_{i}^{\pm}}(r) > \frac{3}{8m}|r| = \frac{3}{8m}\ell \geq \frac{5}{4m}\lceil \frac{\ell}{4}\rceil = \frac{5}{4m}|u_{j}| > \#_{a_{i}^{\pm}}(u_{j})$$ hold for all $i$ is bounded below by 
    $$1 - 2me^{-C_{m}\lceil \frac{\ell}{4}  \rceil+1} - 2me^{-C_{m}\ell +1}.$$ As a consequence, the probability that $\#_{a_{i}^{\pm}}(u_{j}) \geq \frac{1}{2}\#_{a_{i}^{\pm}}(r)$ is bounded above by 
    $$2me^{-C_{m}\lceil \frac{\ell}{4}  \rceil+1} + 2me^{-C_{m}\ell +1}.$$ Applying the union bound to account for all the different $u_{j}$, the probability we wish to control is bounded above by 
    $$\ell\left(2me^{-C_{m}\lceil \frac{\ell}{4}  \rceil+1} + 2me^{-C_{m}\ell +1}\right) \lesssim \exp(-\frac{C_{m}}{4}\ell).$$
\end{proof}

\begin{lem}\label{lem: reducedwordevendistribution3}
    Let $r$ be a reduced word of length $\ell$. The probability that some subword $u$ of $r$ with $|u| = \lceil \frac{1}{16} \ell \rceil$ satisfies 
    $$\#_{a_{i}^{\pm}}(u) \leq \frac{1}{8m}|u|$$ for some $i$ is $\lesssim \exp(-\frac{C_{m}}{16}\ell)$.
\end{lem}
\begin{proof}
    Let $u_{j}$ be the subword of $r$ of length $ \lceil \frac{1}{16} \ell \rceil$ starting at the $j$th letter of $r$. By Lemma \ref{lem: reducedwordprobability}, the probability that $\#_{a_{i}^{\pm}}(u_{j})  \leq \frac{1}{8m}|u_{j}| < \frac{3}{4m}|u_{j}|$ for all $i$ is bounded above by 
    $$2me^{-\frac{C_{m}}{16}\ell}.$$ Applying the union bound to account for all the different $u_{j}$, the desired probability bounded above by 
    $$2m\ell e^{-\frac{C_{m}}{16}\ell} \lesssim \exp(-\frac{C_{m}}{16}\ell).$$
\end{proof}

\begin{cor}\label{cor: reducedwordevendistribution}
    Let $r$ be a uniformly chosen reduced word of length $\ell$. The probability that $r$ fails to satisfy the even distribution conditions is $\lesssim \exp(-\frac{C_{m}}{16}\ell)$.
\end{cor}
\begin{proof}
    This follows from Lemma \ref{lem: reducedwordevendistribution1}, Lemma \ref{lem: reducedwordevendistribution2}, Lemma \ref{lem: reducedwordevendistribution3}, Remark \ref{rmk: expbounds}, and the fact that $\frac{1}{16} \ln(2m-1) , \frac{C_{m}}{4} \geq \frac{C_{m}}{16}$ for $m \geq 2$.
\end{proof}

\subsection{Genericity of Translation Apparent Presentations}

Since our random groups involve choosing  cyclically reduced words, we now show the cyclically reduced version of results in the previous section.

\begin{lem}\label{lem: cyclicallyreducedwordprobability}
    Let $r$ be a uniformly chosen cyclically reduced word of length $\ell$. The probability that some cyclic permutation of $r$ or $r^{-1}$ fails to satisfy the even distribution conditions (for $\lambda = \frac{1}{16}$) is $\lesssim \exp(-\frac{C_{m}}{16}\ell)$.
\end{lem}
\begin{proof}
    Note that a cyclically reduced word $r$ which fails the even distribution conditions is also a reduced word $r$ which fails the same conditions. By Corollary \ref{cor: reducedwordevendistribution}, there exists a subexponential function $f(\ell)$ so that among the $2m(2m-1)^{\ell-1}$ reduced words of length $\ell$, at most 
    $$2m(2m-1)^{\ell-1} f(\ell) e^{-\frac{C_{m}}{16}\ell}$$ of them fail the even distribution condition. Therefore, the same upper bound holds for the number of cyclically reduced words of length $\ell$ failing those conditions.

    Note that for a cyclically reduced word $r$ of length $\ell$, there are at most $2\ell$ distinct cyclically reduced words which arise as a cyclic permutation of $r$ or $r^{-1}$. Therefore, the number of cyclically reduced words $r$ of length $\ell$ such that some cyclic permutation of $r$ or $r^{-1}$ fails the even distribution conditions is bounded above by 
    $$4m(2m-1)^{\ell-1}\ell f(\ell)e^{-\frac{C_{m}}{16}\ell}.$$

    As shown in \cite[Proposition 17.2]{mann}, the number of cyclically reduced words of length $\ell$ is $(2m-1)^{\ell}+ m + (-1)^{\ell}(m-1) \geq (2m-1)^{\ell}$. Therefore, the probability we wish to control is bounded above by 
    $$\frac{4m}{2m-1} \ell f(\ell) e^{-\frac{C_{m}}{16}\ell} \lesssim \exp(-\frac{C_{m}}{16}\ell).$$
\end{proof}

Finally, we prove Theorem \ref{thm: mainthm2} and Theorem \ref{thm: mainthm3}.

\begin{thm}[Theorem \ref{thm: mainthm2}]
    For $m \geq 2$, let $d_{m} = \frac{C_{m}}{16\ln(2m-1)}$. Then for any density $0 \leq d < d_{m}$, a generic random group on $m$ letters at density $d$ admits a $\frac{1}{16}$-translation-apparent presentation. In fact, this presentation can be taken to be the symmetrization of the random presentation for the group. 
\end{thm}
\begin{proof}
    Fix $\ell$, and let $\langle A | \mathcal{R}_{\ell} \rangle$ denote the random presentation for the random group $G_{\ell}$. Here, $A = \{a_{1}, \ldots, a_{m}\}$ and $\mathcal{R}_{\ell}$ consists of a set of at most $(2m-1)^{d\ell}$ uniformly chosen cyclically reduced words of length $\ell$. Consider the symmetrization $\langle A | \mathcal{R}_{\ell}^{*}\rangle$, obtained by adding in all cyclic permutations of words in $\mathcal{R}_{\ell}$ and their inverses.

    It is known that a generic random group at density $d < \frac{1}{2}$ is infinite, torsion free, non-elementary hyperbolic, and satisfies the $C'(2d + \eps)$ small cancellation condition for any $\eps > 0$ \cite[9.B]{gromovrandom}\cite[V.d]{olliviersurvey}. Since in our case $d_{m} < \frac{1}{32}$, we see that a generic random group at density $d < d_{m}$ is infinite, torsion free, non-elementary hyperbolic, and satisfies the $C'(\frac{1}{16})$ small cancellation condition. Therefore, it only remains to show that the even distribution condition holds generically.

    By Lemma \ref{lem: cyclicallyreducedwordprobability}, the probability that a uniformly chosen cyclically reduced word $r$ of length $\ell$ has some cyclic permutation of itself or $r^{-1}$ which fails the even distribution conditions is $\lesssim \exp(-\frac{C_{m}}{16}\ell)$. In other words, there exists a subexponential function $f(\ell)$ (distinct from previous uses of $f(\ell)$) so that this probability is bounded above by 
    $$f(\ell) e^{-\frac{C_{m}}{16}\ell}.$$
    
    There are at most $(2m-1)^{d\ell}$ uniformly chosen cyclically reduced words of length $\ell$ in $\mathcal{R}_{\ell}$, so by the union bound the probability that some $r \in \mathcal{R}_{\ell}^{*}$ fails the even distribution conditions is bounded above by 
    $$(2m-1)^{d\ell}f(\ell)e^{-\frac{C_{m}}{16}\ell}.$$ Since $f(\ell)$ is subexponential in $\ell$, and since $d\ln(2m-1) < d_{m}\ln(2m-1) =  \frac{C_{m}}{16}$, we see that this upper bound converges to $0$ as $\ell \to \infty$. This proves that the even distribution conditions are generically satisfied by $\langle A | \mathcal{R}_{\ell}^{*}\rangle$, and so admitting a $\frac{1}{16}$-translation-apparent presentation is generic at density $d$.
\end{proof}

Before proving Theorem \ref{thm: mainthm3}, we quickly show a helpful lemma regarding weights on the free group. This lemma is a consequence of a more general stability result (Proposition \ref{prop: stability}), which we prove in the next section.

\begin{lem}\label{lem: freestability}
    Fix $\eps > 0$. There exists $\delta > 0$ so that if $w$ is a normalized weight on $F_{m}$ with the standard generating set, and $h(F_{m}, w) < h(F_{m}, w_{*}) + \delta$, then
    $$\norm{w - w_{*}} < \eps$$ where $w_{*}$ is the uniform weight. 
\end{lem}
\begin{proof}
    Suppose, for a contradiction, that no such $\delta$ exists. Then we can find a sequence of normalized weights $(w_{i})_{i \geq 1}$ and a sequence $(\delta_{i})_{i \geq 1}$ so that $\delta_{i} \to 0$, $h(F_{m}, w_{i}) < h(F_{m}, w_{*}) + \delta_{i}$, and $\norm{w_{i} - w_{*}} \geq \eps$. Since $w_{*}$ is the unique normalized minimizer for $F_{m}$, we have $h(F_{m}, w_{i}) \geq h(F_{m}, w_{*})$ and so $h(F_{m}, w_{i}) \to h(F_{m}, w_{*})$. By Proposition \ref{prop: stability} we have $\norm{w_{i}-w_{*}} \to 0$, contradiction.
\end{proof}

\begin{thm}[Theorem \ref{thm: mainthm3}]
   Fix $\eps > 0$, and let $d_{m}$ be defined as above. A generic random group $G$ on $m$ letters at density $d < d_{m}$ satisfies $h(F_{m}, w) - \eps \leq h(G, w) \leq h(F_{m}, w)$ for \textit{any} normalized weight $w$. Moreover, $G$ admits a unique normalized minimizing metric  $w_{0}$ satisfying: 
    \begin{enumerate}
        \item $\norm{w_{0}- w_{*}} < \eps$, where $w_{*}$ is the uniform normalized weight on the generators.
        \item The minimum entropy of $G$ is within $\eps$ of the minimum entropy of $F_{m}$.
    \end{enumerate}
\end{thm}
\begin{proof}
  Let $\delta $ satisfy the conditions of Lemma \ref{lem: freestability} for our fixed $\eps$. Note that we can take $\delta < \frac{\eps}{2}$ without loss of generality. Fix $\ell$ sufficiently large so that $\lceil \frac{\ell}{16} \rceil > 32m$, $\frac{2}{\lceil \frac{\ell}{16} \rceil} < \frac{\delta}{2}$, and $8m\lceil \frac{\ell}{16} \rceil\ell(2m-1)^{d\ell} < (2m-1)^{\frac{\lceil \frac{\ell}{16} \rceil}{16m}-2}$. The last assumption is possible since $d < d_{m} < \frac{1}{256m}$. Consider the random presentation $\langle A | \mathcal{R}_{\ell} \rangle$ for the random group $G_{\ell}$. Let us first assume that the symmetrization $\langle A | \mathcal{R}_{\ell}^{*}\rangle$ is a $\frac{1}{16}$-translation-apparent presentation.
  
  Recall from Lemma \ref{lem: generatinginequality} the definition
  $$(\mathcal{R}_{\ell}^{*})_{\frac{1}{16}, s} = \{u |  u \text{ is a subword of } r \in \mathcal{R}_{\ell}^{*} \text{ ending with } s, |u| = \lceil \frac{1}{16}|r| \rceil\}.$$ Further recall the notation from Proposition \ref{prop: integerreducedinequality} and Proposition \ref{prop: normalizedreducedinequality}, where we set $j = \max_{s \in S} |(\mathcal{R}_{\ell})_{\frac{1}{16}, s}|$, and take $l$ to be the shortest unweighted length of a word in $\bigcup_{s \in S} (\mathcal{R}_{\ell}^{*})_{\frac{1}{16}, s}$.

  It is not difficult to see that $j \leq |\mathcal{R}_{\ell}^{*}| \leq \ell(2m-1)^{d\ell}$, and that $l = \lceil \frac{\ell}{16} \rceil$. Now, by our choice of $\ell$, we have $l > 32m$ and  
  $$8mjl \leq 8ml\ell(2m-1)^{d\ell} < (2m-1)^{\frac{l}{16m}-2}.$$ Therefore, by Proposition \ref{prop: normalizedreducedinequality}, we have 
  $$h(F_{m}, w) - \eps < h(F_{m}, w) - \frac{\delta}{2} < h(F_{m}, w) - \frac{2}{l} \leq h(G, w) \leq h(F_{m}, w).$$
  
  Since $\langle A | \mathcal{R}_{\ell}^{*}\rangle$ is a $\frac{1}{16}$-translation-apparent presentation, by Theorem \ref{thm: mainthm1} $(G_{\ell}, S)$ admits a unique  normalized minimizing weight $w_{0}$. Suppose, for a contradiction, that $\norm{w_{0} - w_{*}} \geq \eps$. Then by Lemma \ref{lem: freestability} we have $h(F_{m}, w_{0}) \geq h(F_{m}, w_{*}) + \delta$. But by the inequalities shown above for arbitrary $w$, we also have 
$$h(G, w_{*}) \leq h(F_{m}, w_{*}) \leq h(F_{m}, w_{0}) - \delta < h(G, w_{0}) -\frac{\delta}{2} < h(G, w_{0}),$$ contradicting that $w_{0}$ is a minimizer. Therefore, we have $\norm{w_{0} - w_{*}} < \eps$. 

Finally, note that the previous argument shows that $h(F_{m}, w_{0}) < h(F_{m}, w_{*}) + \delta$, which implies $|h(F_{m}, w_{0}) - h(F_{m}, w_{*})| < \delta$ since $w_{*}$ is a minimizer for $F_{m}$. It follows from the inequalities above that 
$$|h(G, w_{0}) - h(F_{m}, w_{*})| \leq |h(G, w_{0}) - h(F_{m}, w_{0})| + |h(F_{m}, w_{0}) - h(F_{m}, w_{*})| < \frac{\delta}{2} + \delta < \eps.$$ This proves that the minimum entropy of $G_{\ell}$ is within $\eps$ of the minimum entropy of $F_{m}$.

All of the inequalities above hold as long as $\ell$ is sufficiently large and the symmetrization $\langle A | \mathcal{R}_{\ell}^{*}\rangle$ is a $\frac{1}{16}$-translation-apparent presentation. As shown in the proof of Theorem \ref{thm: mainthm2}, this assumption that the symmetrization $\langle A | \mathcal{R}_{\ell}^{*}\rangle$ is a $\frac{1}{16}$-translation-apparent presentation is generically satisfied at density $d < d_{m}$, and so the inequalities above are generically satisfied as well.
\end{proof}

\begin{rmk} For completeness, let us observe that the value $d_{m}$ we have obtained is explicitly given by 
$$d_{m} = \frac{C_{m}}{16\ln(2m-1)} = \frac{(m-1)^{2}}{6144m^{2}(2m-1)^{2}\ln(2m-1)},$$ which goes to $0$ as $m \to \infty$. Note that even for $m = 2$ we have $d_{2} \approx \frac{1}{242995}$, which is much smaller than the usual density bounds which appear in the literature; this is the reason we refer to our situation as a ``low density" case. We have not tried very hard to optimize this value of $d_{m}$; indeed, our choice of $\lambda = \frac{1}{16}$ is arbitrary, and a better choice could improve $d_{m}$. However, we do not anticipate that such changes would result in $d_{m}$ comparable to the typical density scales.
\end{rmk}

\section{Stability}\label{sec: stability}
Once we have rigidity, we can consider the notion of \textit{stability}: if a sequence of objects have volume entropy converging to the minimum entropy, do these objects converge (in some sense) to the unique minimizer? In the context of volume entropy, Song \cite{songentropy} has shown that stability holds for the volume entropy of hyperbolic manifolds, with respect to a slightly modified version of measured Gromov-Hausdorff convergence.

We will show that as long as the space of normalized weights on a marked torsion-free, non-elementary hyperbolic group $(G, S)$ admits a unique minimizing weight, then we have stability with respect to the Euclidean metric on $\mathcal{W}$ (viewed naturally as a subset of $\RR^{|S|}$). 

\begin{prop}\label{prop: stability}
Let $(G, S)$ be a marked torsion-free, non-elementary group, with the property that the space $\mathcal{W}$ of normalized weights on $S$ contains a unique weight $w_{0}$ minimizing the volume entropy. Let $(w_{i})_{i \geq 1}$ be a sequence of weights in $\mathcal{W}$ such that 
$$h(w_{i}) \to h(w_{0}).$$ Then we have $w_{i} \to w_{0}$ in the Euclidean metric on $\mathcal{W}$.
\end{prop}
\begin{proof}
 Consider an arbitrary subsequence $(w_{i_{j}})_{j \geq 1}$ of $(w_{i})_{i \geq 1}$. Now, for each $s \in S$, consider the sequence of real numbers $(w_{i_{j}}(s))_{j \geq 1}$. Since $h(w_{i_{j}}) \to h(w_{0}) < \infty$, by Proposition \ref{prop: compactness} there exists $\delta > 0$ so that for sufficiently large $j$, $w_{i_{j}}(s) > \delta$.

So, $w_{i_{j}}(s)$ is contained in the compact set $[\delta, 1]$ for all sufficiently large $j$, and thus $(w_{i_{j}}(s))_{j \geq 1}$ must contain a convergent subsequence. Since this is true for each $s \in S$, we may pass to a subsequence finitely many times to obtain a subsequence $(w_{i_{j_{k}}})_{k \geq 1}$ of $(w_{i_{j}})_{j \geq 1}$ so that $(w_{i_{j_{k}}}(s))_{k \geq 1}$ converges for each $s$. Define $w_{\infty} \colon S \to (0, \infty)$ by 
$$w_{\infty}(s) = \lim_{k \to \infty} w_{i_{j_{k}}}(s) \geq \delta$$ for each $s$. Since each $w_{i_{j_{k}}}$ is normalized, we see that 
$$\sum_{s \in S} w_{\infty}(s) = \lim_{k \to \infty} \sum_{s \in S}w_{i_{j_{k}}}(s) = 1,$$ hence $w_{\infty} \in \mathcal{W}$ as well. 

By continuity of $h$, we have 
$$h(w_{\infty}) = \lim_{k \to \infty} h(w_{i_{j_{k}}}) = \lim_{i \to \infty} h(w_{i}) = h(w_{0}),$$ and by uniqueness of the minimizer we have $w_{\infty} = w_{0}$. So, every subsequence of $(w_{i})_{i \geq 1}$ contains a further subsequence converging to $w_{0}$ in the Euclidean metric on $\mathcal{W}$, hence $(w_{i})_{i \geq 1}$ must itself converge to $w_{0}$. 
\end{proof}

\section{Questions and Future Work}\label{sec: futurework}
To conclude, we mention some questions representing potential directions for future work. The first is a natural question to ask, since (somewhat embarrassingly) we do not have any counterexamples at hand:
\begin{quest}\label{quest: generalhyperbolic}
    Does every marked torsion-free, hyperbolic group $(G, S)$ have a unique minimizing normalized weight?
\end{quest}
Let us note that the answer is ``no" if we instead ask about strict convexity of the volume entropy. 
\begin{exmp}\label{exmp: nostrictconvexity}
    Consider $A = \{a, b, a^{2}, b^{2}\} \subset F_{2}$, and take the symmetric generating set $S = A \cup A^{-1}$ for $F_{2}$. Let $w_{0}$ be given by $w_{0}(a) = w_{0}(b) = \frac{1}{16},$ $w_{0}(a^{2}) = w_{0}(b^{2})  = \frac{3}{16}$, and let $w_{1}$ be given by $w_{1}(a) = w_{1}(b) = \frac{1}{16}$, $w_{1}(a^{2}) = \frac{1}{8}$, $w_{1}(b^{2}) = \frac{1}{4}$. 

    For $t = [0, 1]$, define $w_{t} = (1-t)w_{0} + tw_{1}$. Each $w_{t}$ is normalized, and moreover for each $t$ we have 
    \begin{align*}
    w_{t}(a^{2}) &\geq 2w_{t}(a),\\
    w_{t}(b^{2}) &\geq 2w_{t}(b).
    \end{align*}
    Therefore, when considering $d_{w_{t}}(e, v)$ for $v \in F_{2}$, there is never a need to use $a^{2}$ or $b^{2}$ instead of two $a$'s or two $b$'s. It follows that for all $t$ we have $h(w_{t}) = h(F, \{a^{\pm}, b^{\pm}\}, w)$, where $w(a) = w(b) = \frac{1}{16}$. Therefore $h$ is not strictly convex. 
\end{exmp}

We also present a potential source of counterexamples. Let $G$ be a hyperbolic group and $S$ a symmetric generating set. Fix an order on the elements of $S$, which extends to a lexicographical order on the set of words in $S$. Each $x \in G$ may be represented by multiple words in $S$, so we select the unique lexicographically first representative for 
each $x$.

Using the theory of cone types, Cannon showed in \cite{cannon} that there exists a directed graph $\Gamma$ so that there is a bijection between the directed paths in $\Gamma$ starting at a fixed basepoint and lexicographically first representatives for elements in $G$. Moreover, this bijection sends directed paths of length $n$ to words with word length $n$. 

Tracing through the proof, we see that for a fixed weight $w$ on $(G, S)$, there is a directed graph $\Gamma_{w}$ with weighted edges, so that there is once again a bijection between directed paths and lexicographically first representatives which additionally preserves weighted length. Therefore, $h(w)$ is equal to the exponential growth rate of the number of directed paths in $\Gamma_{w}$.

Now, as observed by McMullen \cite{mcmullen}, there may not be a unique metric (coming from edge weights) on a directed graph minimizing the growth rate of directed paths; in general there is only a unique optimal cohomology class of metric. Therefore a more careful study of these weighted directed graphs could lead to counterexamples.

Instead of considering all torsion-free, non-elementary hyperbolic groups, we can once again turn to random groups and ask,
\begin{quest}\label{quest: generaldensity}
    Which densities $0 \leq d \leq 1$ have the property that a generic random group in the density $d$ model admits a unique minimizing weight?
\end{quest}
Immediately, let us note that there is no need to consider $d > \frac{1}{2}$: Gromov showed in \cite{gromovrandom} that a generic group in this density regime is either trivial or $\ZZ/2\ZZ$ (due to parity reasons). For $d < \frac{1}{2}$, however, a generic group in the density model is torsion-free and non-elementary hyperbolic, and so this question is non-trivial. It is clear that conditions like translation-apparence fail generically for larger densities under $\frac{1}{2}$, so our approach cannot generalize. A potential idea to consider could be Ollivier's method of decorated Van Kampen diagrams, introduced in \cite{olliviersharp}. Perhaps a theory of weighted decorated Van Kampen diagrams could be of use here.

Another direction of further inquiry would be to allow the generating set to vary, and consider minimizers over all generating sets. \begin{quest}\label{quest: variedgeneratingsets}
    Consider all generating sets $S$ of a group $G$, and for each $S$ consider all normalized weights on $S$. For which groups $G$ does there exist a unique pair $(S, w)$ minimizing $h(G, S, w)$ over all such pairs?
\end{quest}
Observe that for a torsion-free, non-elementary hyperbolic group $G$, the existence of a minimizing pair $(S, w)$ holds by the same arguments used in Proposition \ref{prop: existence}. For roughly related work in this general direction for the unweighted case, involving studying the entropy for varying generating sets, we highlight \cite{fujiwarasela} \footnote{We thank Stephen Cantrell for pointing this reference out to us.} and \cite{wilson}. A generalization of these results to the weighted case could be interesting as well.

Finally, we end with an open ended question:
\begin{quest}\label{quest: geometricapplications}
    Can we apply the methods of this paper to other settings?
\end{quest}
One difficulty is that whereas a convex combination of normalized weights remains normalized, taking convex combinations of other objects may not preserve the properties we want. For example, in the classic setting of Besson-Courtois-Gallot \cite{BessonCourtoisGallot}, we are considering all Riemannian metrics $g$ on a closed hyperbolic manifold $(M, g_{0})$ satisfying $\Vol(M, g) = \Vol(M, g_{0})$, a condition which is not preserved by taking convex combinations of metrics.

\appendix

\section{Counting \texorpdfstring{$\lambda$-reduced words}{}}\label{sec: appendix}

In this appendix we prove the lemma we stated without proof in Section \ref{sec: translationapparent}. The proof is a direct application of the counting method for weighted subword avoidance developed in \cite{myers}. The following definitions are all given in \cite{myers}, although we change some notation for simplicity.

Let $\Omega$ be an alphabet equipped with an integer-valued weight $w$ assigning a weight to each letter, and let $\mathcal{F}$ be a set of forbidden words with letters in $\Omega$. Let us further assume that $\mathcal{F}$ is \textit{reduced}: no word in $\mathcal{F}$ contains another word in $\mathcal{F}$ as a subword. For each $n$, let $f(n)$ denote the number of words of total $w$-weight $n$ which do not contain any element of $\mathcal{F}$. Similarly, for each word $W \in \mathcal{F}$, let $f_{W}(n)$ denote the number of words of total $w$-weight $n$ which end with $W$, and do not contain any other forbidden words as a subword. We are interested in the generating functions $F(z) = \sum_{n \geq 0} f(n)z^{-n}$ and $F_{W}(z) = \sum_{n \geq 0} f_{W}(n)z^{-n}$. Here, we take $f(0) = 1$ (counting the empty word) and $f_{W}(0) = 0$.

Given $W_{1} = x_{1}\cdots x_{r}, W_{2} = y_{1} \cdots y_{t} \in \mathcal{F}$, the \textit{weighted correlation} $w(W_{1}, W_{2})$ is defined by 
$$\{j |  \exists i, x_{i} = y_{1}, x_{i+1} = y_{2}, \ldots, x_{r} = y_{r-i+1}, j = w(y_{1}) + \cdots + w(y_{r-i+1})\}.$$ In other words, $j$ is in $w(W_{1}, W_{2})$ iff there is a word $v$ of total weight $j$ such that $W_{1}$ ends with $v$ and $W_{2}$ begins with $v$. Note that this definition is, in general, not symmetric in $W_{1}, W_{2}$. The \textit{weighted correlation polynomial} $w(W_{1}, W_{2})_{z}$ is given by $\sum_{j \in w(W_{1}, W_{2})} z^{j}$, where we take $w(W_{1}, W_{2})_{z} = 0$ if $w(W_{1}, W_{2})$ is empty.

The main result of Myers is the following system of equations for the generating functions defined above. 
\begin{thm}[{\cite[Theorem 2]{myers}}]\label{thm: myers}
    Let $\mathcal{F} = \{W_{1}, \ldots, W_{k}\}$ be a reduced set of forbidden words over the alphabet $\Omega$ with weight $w$. For ease of notation, we use the abbreviations $w_{ij}(z) \coloneqq w(W_{i}, W_{j})_{z}$, $f_{i}(n) \coloneqq f_{W_{i}}(n)$, and $F_{i}(z) \coloneqq F_{W_{i}}(z)$.
Then the generating functions $F(z), F_{1}(z), \ldots, F_{k}(z)$ satisfy the following system of equations:
\begin{center}
\begin{tabular}{ c c c c c c}
 $\left(1 - \sum_{s \in \Omega}z^{-w(s)} \right)F(z)$ & $+F_{1}(z)$ & $+F_{2}(z)$ & $\cdots$ & $+F_{k}(z)$  & $=1$\\
 $F(z)$ & $-w_{11}(z)F_{1}(z)$ & $-w_{21}(z)F_{2}(z)$ & $\cdots$ & $-w_{k1}(z)F_{k}(z)$  & $= 0$\\
 $F(z)$ & $-w_{12}(z)F_{1}(z)$ & $-w_{22}(z)F_{2}(z)$ & $\cdots$ & $-w_{k2}(z)F_{k}(z)$  & $= 0$\\
 $\vdots$ & $\vdots$ & $\vdots$ & $\vdots$ & $\vdots$ &$\vdots$\\
 $F(z)$ & $-w_{1k}(z)F_{1}(z)$ & $-w_{2k}(z)F_{2}(z)$ & $\cdots$ & $-w_{kk}(z)F_{k}(z)$  & $= 0$\\
\end{tabular}
\end{center}
\end{thm}
\begin{rmk}
    There are some minor errors in Myers' original statement and proof, so we provide a short proof here for completeness.
\end{rmk}
\begin{proof}
   Since $\mathcal{F}$ is reduced, removing the last letter $s$ from a word counted in $f_{i}(n)$ gives a word counted in $f(n - w(s))$. Therefore, we have the equality 
   $$f(n)  + f_{1}(n) + \cdots + f_{k}(n) - \sum_{s \in \Omega} f(n - w(s))= 0$$ for all $n$, except when $n = 0$ the right hand side should be $1$. Here, we interpret $f(a) = f_{i}(a) =  0$ when $a < 0$. Since the left hand side is the coefficient of $z^{-n}$ in $\left(1 - \sum_{s \in \Omega}z^{-w(s)} \right)F(z) + F_{1}(z) + \cdots + F_{n}(z)$, the first desired equality follows.

   Now, for a fixed $W_{i}$, we claim that 
   $$f(n) = \sum_{j_{1} \in w(W_{1}, W_{i})} f_{1}(n+ j_{1}) + \cdots + \sum_{j_{k} \in w(W_{k}, W_{i})} f_{k}(n+ j_{k}).$$ 
   
First let $X = x_{1}\cdots x_{r}$ be a word counted by $f(n)$, so that $X$ has total weight $n$. Append $W_{i}$ to the end of $X$. Then $XW_{i}$ can be uniquely written as a concatenation $YZ$, where $Y$ ends with some $W_{\ell}$ (and contains no other forbidden words), and $Z$ is possibly empty. Note that since $X$ does not contain any forbidden words, we can further write $Y = XV$, where $W_{\ell}$ ends with the subword $V$. We see that $VZ = W_{i}$, and we conclude that $|V|_{w} \in w(W_{\ell}, W_{i}).$ Therefore, we see that $Y$ is counted in $\sum_{j_{\ell} \in w(W_{\ell}, W_{i})} f_{\ell}(n + j_{\ell})$. Note that if we begin with distinct $X, X'$, the $Y, Y'$ obtained through this process are also distinct. This is because if $Y = Y'$, then writing $Y = XV$, $Y' = X'V'$, we see that $V, V'$ have the same total weight and that $W_{\ell}$ ends with both $V$ and $V$'. Therefore $V = V'$, and so $X = X'$. This shows that each word counted on the left hand side is counted uniquely on the right hand side. 

For the other direction, suppose that $Y$ is counted in $f_{\ell}(n + j_{\ell})$ for some $\ell$ and $j_{\ell} \in w(W_{\ell}, W_{i})$. It is clear that this can be true for only one $\ell$ and one $j_{\ell} \in w(W_{\ell}, W_{i})$. By definition, $Y$ ends with $W_{\ell}$, and there exists some $V$ with $|V|_{w} = j_{\ell}$ such that $W_{\ell}$ ends with $V$, and $W_{i}$ begins with $V$. Now, we can write $Y = XV$, where $X$ does not contain any forbidden subwords, so that $X$ is counted in $f(n)$. As above, distinct $Y, Y'$ correspond to distinct $X, X'$. This shows that each word counted on the right hand side is counted uniquely on the left hand side, proving the claimed equality.

The left hand side above is the coefficient of $z^{-n}$ in $F(z)$, and the right hand side is the coefficient of $z^{-n}$ in $\sum_{\ell=1}^{n} w_{\ell i}(z)F_{W_{\ell}}(z)$. This proves the remaining inequalities. 
\end{proof}

As Myers notes in \cite{myers}, the fact that $\mathcal{F}$ is reduced implies that the matrix of coefficients of the above system is invertible, and so we conclude that $F(z), F_{1}(z), \ldots, F_{k}(z)$ are all rational functions of $z$.

Now we apply this result to proving the lemma in Section \ref{sec: translationapparent}. Let $A = \{a_{1}, \ldots, a_{m}\}$, and let $G = \langle A | \mathcal{R}\rangle$ be a $\lambda$-translation-apparent presentation. We are interested in counting $\lambda$-reduced words, which are the reduced words $x$ over $A$ so that any common subword $u$ of $x$ and $r \in \mathcal{R}$ satisfies $|u| < \lceil \lambda |r|\rceil$. 

In the above framework, we take $\Omega = A \cup A^{-1}$. Since $\lambda$-reduced words are reduced, each word of the form $A_{i}^{+} = a_{i}a_{i}^{-1}$ and $A_{i}^{-} = a_{i}^{-1}a_{i}$ must be in our set of forbidden words $\mathcal{F}$. Moreover, for each $r \in \mathcal{R}$, any subword $u$ of $r$ satisfying $|u| = \lceil \lambda |r| \rceil$ must be in $\mathcal{F}$. Since $\langle A | \mathcal{R}\rangle$ is $\lambda$-translation-apparent, $\mathcal{R}$ is symmetrized and consists of cyclically reduced words, hence to each $r$ we can associate a unique forbidden word $U_{r}$ by taking the first $\lceil \lambda |r| \rceil$ letters of $r$ (all other such subwords of $r$ are included via taking the first letters of some cyclic permutation of $r$). The subwords described above are the only ones that must be avoided, so we obtain
$$\mathcal{F} = \{A_{1}^{\pm}, \ldots, A_{m}^{\pm}, U_{r} |  r \in \mathcal{R}\}.$$ Note that this is reduced: no $U_{r}$ contains any $A_{i}^{\pm}$ since $\mathcal{R}$ consists of reduced words, and no $U_{r}$ contains another $U_{r'}$ since $\langle A | \mathcal{R}\rangle$ satisfies the $C'(\lambda)$ small cancellation condition.

Now let $w$ be an integer valued weight on $\Omega$. For each $i$, let $\mathcal{U}_{i^{+}}$ denote the set of forbidden words of the form $U_{r}$ which end with $a_{i}$, and similarly define $\mathcal{U}_{i^{-}}$. For simplicity, write $w_{i} = w(a_{i}^{\pm})$. Now, calculating the weighted correlation polynomials, we see that $w(A_{i}^{+}, A_{i}^{+})_{z} =w(A_{i}^{-}, A_{i}^{-})_{z} = z^{2w_{i}}$, $w(A_{i}^{+}, A_{i}^{-})_{z} = w(A_{i}^{-}, A_{i}^{+})_{z} = z^{w_{i}}$, and $w(A_{i}^{\pm}, A_{j}^{\pm})_{z} = 0$ for $i \neq j$. Moreover, we have 
$w(U_{r}, A_{i}^{+})_{z} = z^{w_{i}}$ if $U_{r} \in \mathcal{U}_{i^{+}}$, and $w(U_{r}, A_{i}^{+})_{z} = 0$ otherwise (similarly with $i^{-}$).

Applying Theorem \ref{thm: myers}, we obtain 
\begin{equation}
    \left(1-\sum_{i=1}^{m}2z^{-w_{i}}\right)F(z) + \sum_{i=1}^{m}F_{A_{i}^{+}}(z) + \sum_{i=1}^{m}F_{A_{i}^{-}}(z) + \sum_{r\in  \mathcal{R}}F_{U_{r}}(z) = 1,
\end{equation}
and for each $i$ we have 
\begin{equation}
    F(z) - z^{2w_{i}}F_{A_{i}^{+}}(z) - z^{w_{i}}F_{A_{i}^{-}}(z) - z^{w_{i}}\sum_{U_{r} \in 
 \mathcal{U}_{i^{+}}}F_{U_{r}}(z)= 0,
\end{equation}
\begin{equation}
  F(z) - z^{2w_{i}}F_{A_{i}^{-}}(z) - z^{w_{i}}F_{A_{i}^{+}}(z) - z^{w_{i}}\sum_{U_{r} \in \mathcal{U}_{i^{-}}}F_{U_{r}}(z)= 0.
\end{equation}
Note that there are additional equations with coefficients of the form $w(A_{i}^{\pm}, U_{r})_{z}, w(U_{r}, U_{r'})_{z}$, but we will not need these.

Summing (A.4)and (A.5) and rearranging, we have 
\begin{equation}
    F_{A_{i}^{+}}(z) + F_{A_{i}^{-}}(z) = \frac{2}{z^{w_{i}}(z^{w_{i}}+1)}F(z) - \frac{1}{z^{w_{i}}+1}\sum_{U_{r} \in \mathcal{U}_{i^{\pm}}}F_{U_{r}}(z)
\end{equation} for each $i$, where $\mathcal{U}_{i^{\pm}} = \mathcal{U}_{i^{+}} \cup \mathcal{U}_{i^{-}}$. Summing (A.6) across all $i$, we obtain 
\begin{equation}
    \sum_{i=1}^{m}F_{A_{i}^{+}}(z) + \sum_{i=1}^{m}F_{A_{i}^{-}}(z) = \left(\sum_{i=1}^{m}\frac{2}{z^{w_{i}}(z^{w_{i}}+1)}\right)F(z) - \sum_{i=1}^{m}\frac{1}{z^{w_{i}}+1}\sum_{U_{r} \in \mathcal{U}_{i^{\pm}}}F_{U_{r}}(z).
\end{equation}

Substituting back into (A.3), we have
\begin{equation}
    \left( 1 - \sum_{i=1}^{m}2(z^{w_{i}}+1)^{-1}\right)F(z) - \sum_{i=1}^{m}\frac{1}{z^{w_{i}}+1}\sum_{U_{r} \in \mathcal{U}_{i^{\pm}}}F_{U_{r}}(z) +  \sum_{r \in \mathcal{R}}F_{U_{r}}(z) = 1.
\end{equation}
Finally, using the fact that each $U_{r}$ is contained in $\mathcal{U}_{i^{\pm}}$ for exactly one value of $i$, we have $\sum_{r \in \mathcal{R}}F_{U_{r}}(z) = \sum_{i=1}^{m}\sum_{U_{r} \in \mathcal{U}_{i^{\pm}}}F_{U_{r}}(z)$, and so
\begin{equation}
    \left( 1 - \sum_{i=1}^{m}2(z^{w_{i}}+1)^{-1}\right)F(z) + \sum_{i=1}^{m}\frac{z^{w_{i}}}{z^{w_{i}}+1}\sum_{U_{r} \in \mathcal{U}_{i^{\pm}}}F_{U_{r}}(z) = 1.
\end{equation}

Now we estimate $F_{U_{r}}(z)$. Each word counted by $f_{U_{r}}(n)$ gives a word counted by $f(n - |U_{r}|_{w})$ after removing the copy of $U_{r}$ at the end, and so we have $f_{U_{r}}(n) \leq f(n - |U_{r}|_{w})$ for each $n$. Therefore, we have 
$$F_{U_{r}}(z) \leq z^{-|U_{r}|_{w}}F(z)$$ for any positive, real $z$ within the domain of convergence of $F(z)$. (Note that such a $z$ must lie in the domain of convergence of $F_{U_{r}}(z)$ as well, by the same inequality.) It follows from (A.9) that 
\begin{equation}
    \left( 1 - \sum_{i=1}^{m}2(z^{w_{i}}+1)^{-1}\right)F(z) + \left(\sum_{i=1}^{m}\frac{z^{w_{i}}}{z^{w_{i}}+1}\sum_{U_{r} \in \mathcal{U}_{i^{\pm}}}\frac{1}{z^{|U_{r}|_{w}}}\right)F(z) \geq 1
\end{equation}
for any positive, real $z$ within the domain of convergence of $F(z)$.

Consider the substitution $x = \frac{1}{z}$, so that we have a power series $\sum_{n \geq 0} f(n)x^{n}$ with nonnegative exponents. Since this is a rational function, we know that the growth rate $M = \lim_{n \to \infty} \sqrt[n]{f(n)}$ is equal to $\frac{1}{R}$, where $R$ is the radius of convergence of $\sum_{n \geq 0} f(n)x^{n}$ \cite[Theorem IV.7]{flajoletsedgewick}. Naturally, this implies that $F(z)$ converges for all $|z| > M$. We can now prove the lemma in Section \ref{sec: translationapparent}.

\begin{proof}[Proof of Lemma \ref{lem: generatinginequality}]
    In the notation of Lemma \ref{lem: generatinginequality}, $\mathcal{R}_{\lambda, a_{i}} = \mathcal{U}_{i^{+}}$, and similarly for $i^{-}$. Suppose, for a contradiction, that $p(z)$ has a real root $z = z_{0}$ with $z_{0} > M$. Then $z_{0}$ lies within the domain of convergence of $F(z)$, and by (A.10) we have 
    $$0 = p(z_{0})F(z_{0}) \geq 1,$$ contradiction.
\end{proof}

\bibliographystyle{amsplain}
\bibliography{main}

\end{document}